\documentclass[5p,twocolumn,times,authoryear,a4paper]{autart}     
\usepackage[utf8]{inputenc}
\usepackage{amsmath}
\usepackage[export]{adjustbox}
\usepackage{mathtools}
\usepackage{amssymb}
\usepackage{color,soul}
\usepackage{cite}
\usepackage{enumitem}
\usepackage{comment}
\usepackage{graphicx}

\newtheorem{theorem}{Theorem}   
\newtheorem{lemma}[theorem]{Lemma}
\newtheorem{problem}{Problem}
\newtheorem{proposition}[theorem]{Proposition}
\newtheorem{corollary}[theorem]{Corollary}
\newtheorem{example}{Example}
\newtheorem{remark}{Remark}

\newtheorem{definition}{Definition}
\newtheorem{assumption}{Assumption}
\newenvironment{proof}{\textbf{Proof:}}{\hfill$\square$}
\DeclareMathOperator{\sos}{SOS}

\begin{document}

\begin{frontmatter}

\title{Data-driven stabilization of polynomial systems using density functions\thanksref{footnoteinfo}} 

\thanks[footnoteinfo]{This paper was not presented at any IFAC 
meeting. Corresponding author Huayuan Huang. This work was supported by the China Scholarship Council 202106340040. Henk van Waarde acknowledges financial support by the Dutch Research Council (NWO) under the Talent Programme Veni Agreement (VI.Veni.22.335).}

\author[BI]{Huayuan Huang}\ead{huayuan.huang@rug.nl},              
\author[BI]{M. Kanat Camlibel}\ead{m.k.camlibel@rug.nl},
\author[BI]{Raffaella Carloni}\ead{r.carloni@rug.nl},    
\author[BI]{Henk J. van Waarde}\ead{h.j.van.waarde@rug.nl}

\address[BI]{Bernoulli Institute for Mathematics, Computer Science and Artificial Intelligence, Faculty of Science and Engineering, University of Groningen, 9747 AG Groningen, The Netherlands}     
          

\begin{abstract}                          
This paper studies data-driven stabilization of a class of unknown polynomial systems using data corrupted by bounded noise. Existing work addressing this problem has focused on designing a controller and a Lyapunov function so that a certain state-dependent matrix is negative definite, which ensures asymptotic stability of all closed-loop systems compatible with the data. However, as we demonstrate in this paper, considering the negative definiteness of this matrix introduces conservatism, which limits the applicability of current approaches. To tackle this issue, we develop a new method for the data-driven stabilization of polynomial systems using the concept of density functions. The control design consists of two steps. Firstly, a dual Lyapunov theorem is used to formulate a sum of squares program that allows us to compute a rational state feedback controller for all systems compatible with the data. By the dual Lyapunov theorem, this controller ensures that the trajectories of the closed-loop system converge to zero for \textit{almost all} initial states. Secondly, we propose a method to verify whether the designed controller achieves asymptotic stability of all closed-loop systems compatible with the data. Apart from reducing conservatism of existing methods, the proposed approach can also readily take into account prior knowledge on the system parameters. A key technical result developed in this paper is a new type of S-lemma for a specific class of matrices that, in contrast to the classical S-lemma, avoids the use of multipliers.
\end{abstract}
\end{frontmatter}

\section{Introduction}
Research on data-driven control design for unknown dynamical systems has gained significant attention \cite{DDCZiegler1942optimum,iterativetuning,DDCVRFT,DDCStuggatt2020,DDCInformativity,DDCmatrixSlemma,DDCDensity}. This research can be approached from different angles, such as combining system identification and model-based control, or designing control laws directly from the data without the intermediate modeling step. Skipping the step of system identification can offer potential advantages, especially when the system cannot be modeled uniquely \cite{DDCInformativity}. In this study, we will focus on the latter approach, known as direct data-driven control.

Direct data-driven control methods are diverse, including intelligent proportional-integral-derivative control \cite{iPID}, data-driven optimal control \cite{LQR,databasedLQG} and learning-based approaches \cite{DDCLearningbased}. Some methods generate controllers that are updated online \cite{iterativetuning}, while others construct controllers using offline data \cite{DDCFormulas,DDCInformativity}. Inspired by Willems' fundamental lemma \cite{DDCJWilliems}, several recent works have contributed to the latter category. Specifically, for linear systems, a closed-loop system parametrization method \cite{DDCFormulas} has been derived for stabilization and optimal control using persistently exciting data \cite{DDCJWilliems}. This method synthesizes controllers through solving data-based linear matrix inequalities (LMIs). Results in \cite{DDCInformativity} reveal that persistency of excitation is not necessary for certain analysis and control problems, including stabilization. In some cases, data-driven analysis and control is possible, while the data are not informative for system identification \cite{DBLSCT2025}. Under the same framework of data informativity, \cite{DDCmatrixSlemma} investigates another crucial aspect in data-driven control, namely the presence of noise affecting the measured data. In \cite{DDCmatrixSlemma}, a non-conservative design method for stabilizing linear systems using noisy data is developed via a matrix generalization of the classical S-lemma \cite{SurveySlemma}. For nonlinear systems, data-driven control with noisy data is studied in \cite{DDCmiddleTAC,DDCBeyondpolynomial,DDCDensity}. Some approaches rely on a so-called ``linear-like" form, inspired by model-based methods \cite{Originalmiddle}. By incorporating a relaxation method based on a sum of squares (SOS) decomposition \cite{OriginalSOSThesis}, controller design methods for polynomial \cite{DDCmiddleTAC} and rational systems \cite{DDCBeyondpolynomial} are presented by using noisy data. However, the approach taken in \cite{DDCmiddleTAC} has inherent limitations, as we will discuss in Section \ref{section Connection to previous work} of this paper. A different strategy based on Rantzer's dual Lyapunov theorem \cite{DensityDual} and Farkas' Lemma has been proposed for nonlinear systems in \cite{DDCDensity}. When applied to polynomial systems, however, this approach requires a large number of SOS constraints, depending on the product of the state-space dimension and the time horizon of the experiments. In addition, the method only provides theoretical guarantees for the convergence of the state to the origin for \emph{almost all} initial states but not for asymptotic stability, cf. \cite{DensityDual}.

In this study, we focus on the problem of data-driven stabilization of polynomial systems, assuming that the data are corrupted by unknown but bounded noise. Although existing approaches \cite{DDCmiddleCDC,DDCmiddleTAC} are effective in some cases, there are many examples in which they are unable to produce a stabilizing controller. In fact, as we demonstrate in Section 3, this conservatism is due to the fact that existing methods impose that a certain state-dependent matrix is negative definite. The negative definiteness of this matrix is sufficient, but not necessary for the existence of a stabilizing controller. Motivated by these observations, in this paper we propose a new method for data-driven stabilization of polynomial systems that avoids working with this state-dependent matrix. Instead, one of the core concepts used in this paper is the notion of density functions \cite{DensityDual}.

Inspired by \cite{DDCmiddleCDC,DDCmiddleTAC}, we consider input-affine polynomial systems where the vector field and input matrix consist of unknown linear combinations of given polynomials. We assume that some of the parameters of the system are known, while the remaining ones are unknown. This prior knowledge is motivated by physical systems, where known parameters may arise, for example, because variables that represent velocities are derivatives of positions. Within this setup, we represent the set of systems that are compatible with the given input-state data and the prior knowledge by a quadratic inequality. Our aim is to design a robust state feedback controller stabilizing all systems within this set. The control design consists of two steps. First, the dual Lyapunov theorem \cite{DensityDual} is applied to compute a rational control function which ensures that the trajectories of all closed-loop systems compatible with the data and the prior knowledge converge to zero for \textit{almost all} initial states. Second, the asymptotic stability of all closed-loop systems is verified using a common Lyapunov function. To ensure that conditions are satisfied for all systems in the set, we propose a specialized version of the S-lemma, which presents necessary and sufficient conditions for a linear inequality to be implied by a quadratic one. Compared to the classical S-lemma \cite{SurveySlemma}, our version eliminates the need for introducing a scalar variable, known as a multiplier. In addition, for both the control design and the verification of asymptotic stability, we provide computationally tractable approaches using SOS programming.

Our main contributions are the following:
\begin{enumerate}
    \item We propose a new data-driven control design method for a class of polynomial systems, where the novelty lies in formulating the control framework on the basis of density functions.
    \item As the backbone of this method, we establish a novel technical result that provides necessary and sufficient conditions under which a quadratic inequality implies a \emph{linear} one.
\end{enumerate}
Compared to the state-of-the-art, our method has the following benefits:
\begin{enumerate}[label=(\roman*)]
    \item The proposed stabilization approach radically differs from methods in \cite{DDCmiddleCDC, DDCmiddleTAC} that work with the negative definiteness of a state-dependent matrix. We show that the proposed method can produce stabilizing controllers in examples to which these previous methods are not applicable (see Section~\ref{section Connection to previous work}). 
    \item Compared to \cite{DDCDensity}, the size of our SOS constraints does not depend on the data length. We also provide LMI conditions under which the closed-loop system is asymptotically stable (for all initial states).
    \item In contrast to all previous work on data-driven control of polynomial systems, the proposed approach can readily take into account prior knowledge on the system parameters. This contributes to reducing the conservatism of data-driven control of polynomial systems and is also useful in physical systems where certain system parameters are given.
    \item In comparison to the classical S-lemma \cite{SurveySlemma}, our main technical result has the advantage that it does not require a multiplier. This is shown to be beneficial from a computational point of view in the context of stabilization of polynomial systems. 
\end{enumerate}

This article is organized as follows. In Section \ref{section Notation and Problem formulation}, we formulate the problem, followed by a discussion on previous work in Section \ref{section Connection to previous work}. We present our main results in Section \ref{section Methodology}. Then, in Section \ref{section Illustrative examples}, we validate our method through three illustrative examples. The new S-lemma and the proofs of our main results are given in Section \ref{section Proofs}. Finally, Section \ref{section Conclusion} concludes this article. 

\section{Notation and problem formulation}
\label{section Notation and Problem formulation}
\subsection{Notation}
The Euclidean norm of a vector $x\in \mathbb{R}^n$ is denoted by $\|x\|$. For a set of indices $\alpha \subseteq \{1, 2, \dots, p\}$, we define $v_{\alpha}$ as the subvector of $v \in \mathbb{R}^p$ corresponding to the entries indexed by $\alpha$, and $A_{\alpha \bullet}$ as the submatrix of $A \in \mathbb{R}^{p \times q}$ formed by the rows indexed by $\alpha$. For a matrix $Y = \begin{bmatrix} y_1 & y_2 & \cdots & y_q \end{bmatrix}$, where $y_i \in \mathbb{R}^p$ for $i = 1,2, \dots, q $, the vectorization of $Y$ is defined as $\text{vec}(Y) = \begin{bmatrix}
	y_1^{\top} & y_2^{\top} & \cdots & y_q^{\top}
\end{bmatrix}^{\top}$. The Kronecker product of two matrices $A\in \mathbb{R}^{n\times m}$ and $B\in \mathbb{R}^{p\times q}$ is denoted by $A \otimes B\in \mathbb{R}^{np\times mq}$. The space of real symmetric $n\times n$ matrices is denoted by $\mathbb{S}^n$. A matrix $A \in \mathbb{S}^n$ is called positive definite if $x^{\top}Ax>0$ for all $x\in \mathbb{R}^n \setminus \{0\}$ and positive semidefinite if $x^{\top}Ax \geq 0$ for all $x\in \mathbb{R}^n$. This is denoted by $A>0$ and $A \geq 0$, respectively. Negative semidefiniteness and negative definiteness is defined in a similar way and denoted by $A < 0$ and $A \leq 0$, respectively. For a matrix $A \geq 0$, there exists a unique $B\geq 0$ satisfying $A = BB$, and we define $A^{\frac{1}{2}} = B$.

The set of all multivariate polynomials with coefficients in $\mathbb{R}^{p\times q}$ in the variables $x_1,x_2,\dots,x_n$ is denoted by $\mathbb{R}^{p\times q}[x]$, where $x = \begin{bmatrix}
    x_1&x_2& \cdots &x_n
\end{bmatrix}^{\top}$. When $q=1$ we simply write $\mathbb{R}^{p}[x]$, and when $p=q=1$ we use the notation $\mathbb{R}[x]$. We define the set of functions
$$
    \mathcal{V}:= \{ V \in \mathbb{R}[x] \mid  V(0) = 0 \text{ and } V(x) > 0\ \forall x\in \mathbb{R}^n \setminus\{ 0\}\}.
$$
A multivariate polynomial $Q \in \mathbb{R}^{q\times q}[x]$ is called a SOS polynomial if there exists a $\hat{Q} \in \mathbb{R}^{p\times q}[x]$ such that $Q(x) = \hat{Q}^{\top}(x)\hat{Q}(x)$. The set of all $q\times q$ SOS polynomials is denoted by $\sos^q[x]$. For $q=1$, we simply write $\sos[x]$.

\subsection{Problem formulation}
\label{section Problem formulation}
Consider the polynomial input-affine system 
\begin{equation} \label{realsys}
    \dot{x}(t) = A_{\textnormal{true}} F(x(t)) + B_{\textnormal{true}} G(x(t))u(t)+w(t),
\end{equation}
where $x(t)\in \mathbb{R}^n$ is the state, $u(t)\in \mathbb{R}^m$ is the input, and $w(t)\in \mathbb{R}^n$ is the unknown noise term. We assume that $A_{\textnormal{true}} \in \mathbb{R}^{n \times f}$ and $B_{\textnormal{true}} \in \mathbb{R}^{n \times g}$ are unknown, but $F \in \mathbb{R}^{f}[x]$ and $G \in \mathbb{R}^{g\times m}[x]$ are given. For a positive integer $T$ and $t_0 < t_1 < \cdots < t_{T-1}$, consider the input samples $u(t_0),u(t_1),\dots,u(t_{T-1})$, the state samples $x(t_0),x(t_1),\dots,x(t_{T-1})$ and the state derivative samples $\dot{x}(t_0),\dot{x}(t_1),\dots,\dot{x}(t_{T-1})$.
\begin{remark}
    We note that our approach relies on measurements of state derivatives. In some cases, these derivatives may be measured directly, as is the case, for instance, for mechanical systems. When state derivatives cannot be measured directly, there are several approaches to obtain state derivative data from input-state data. In \cite{Ohtasamplingdata}, a sampling framework has been introduced to calculate exact state derivative samples from state measurements. Moreover, \cite{Stuttgartdiscretizations} and \cite{PaoloOrthogonal} propose methods based on discretizations and orthogonal polynomial bases, respectively, to approximate state derivatives. The latter work also derives error bounds that quantify the accuracy of the estimates.
\end{remark}
In addition, consider the noise samples $w(t_0),w(t_1), \dots,\allowbreak w(t_{T-1})$. Then, we have
\begin{equation}\label{realdataeq}
    \dot{\mathcal{X}} = A_{\textnormal{true}} \mathcal{F} + B_{\textnormal{true}} \mathcal{G} \mathcal{U} + \mathcal{W},
\end{equation}
where
$$
    \begin{aligned}
	\dot{\mathcal{X}} &:= \begin{bmatrix} 
            \dot{x}(t_0) & \dot{x}(t_1) & \cdots & \dot{x}(t_{T-1})
	\end{bmatrix},\\
		\mathcal{F} &:= \begin{bmatrix}
			F(x(t_0)) & F(x(t_1)) & \cdots & F(x(t_{T-1}))
		\end{bmatrix},\\
		\mathcal{G} &:= \begin{bmatrix}
			G(x(t_0)) & G(x(t_1)) & \cdots & 	G(x(t_{T-1}))
		\end{bmatrix},\\
		\mathcal{U} &:= \begin{bmatrix}
				u(t_0) & 0 & \cdots & 0\\
				0 & u(t_1) & \cdots & 0\\
				\vdots & \vdots & \ddots & \vdots\\
				0 & 0 & \cdots & u(t_{T-1})
			\end{bmatrix},\\
		\mathcal{W} &:= \begin{bmatrix}
				w(t_0) & w(t_1) & \cdots & w(t_{T-1})
			\end{bmatrix}.
            \end{aligned}
$$
We also define
$$
	\mathcal{X} := \begin{bmatrix}
		x(t_0) & x(t_1) & \cdots & x(t_{T-1})
	\end{bmatrix}.
$$
Although the noise matrix $\mathcal{W}$ is unknown, we assume that it satisfies the quadratic inequality
\begin{equation} \label{noisebound}
	\begin{bmatrix}
		1 \\ \text{vec}(\mathcal{W}^{\top})
	\end{bmatrix}^{\top} 
		\Phi \begin{bmatrix}
		1\\ \text{vec}(\mathcal{W}^{\top})
	\end{bmatrix} \geq 0,
\end{equation}
where
\begin{equation} \label{Phi}
    \Phi := \begin{bmatrix}
	\Phi_{11} & 0\\
	0 & -I
\end{bmatrix}
\end{equation}
is a given matrix with $\Phi_{11} \geq 0$. The noise model in \eqref{noisebound} is equivalent to the energy bound
\begin{equation} \label{noiseboundenergy}
    \sum_{j=1}^n \sum_{i=0}^{T-1}w_j^2(t_i) \leq \Phi_{11}.
\end{equation} 
If the noise samples at every time instant are bounded in norm, that is,
$$
	\|w(t_i)\|^2 \leq \omega \quad \forall i = 0,1,\dots,T-1,
$$
then \eqref{noiseboundenergy} holds with $\Phi_{11} = \omega T$. 

\begin{remark}
It can be verified that the approach of this paper can be extended to the case that
$$
\Phi = \begin{bmatrix}
    \Phi_{11} & \Phi_{12}\\
    \Phi_{21} & \Phi_{22}
\end{bmatrix},
$$
provided that $\Phi_{22} < 0$ and $\Phi | \Phi_{22} \geq 0$, where $\Phi | \Phi_{22}:= \Phi_{11} - \Phi_{12}\Phi_{22}^{-1}\Phi_{21}$ denotes the Schur complement of $\Phi$ with respect to $\Phi_{22}$. For simplicity, throughout this paper, we only consider the specific $\Phi$, as defined in \eqref{Phi}.
\end{remark}

We denote the set of all systems compatible with the data $(\dot{\mathcal{X}},\mathcal{X},\mathcal{U})$ by $\Sigma_d$:
\begin{equation}
\begin{aligned}
    \Sigma_d := \{(A,B)\mid \dot{\mathcal{X}} = & A \mathcal{F} + B \mathcal{G} \mathcal{U} + \mathcal{W}\text{ holds }\\
    & \text{for some }\mathcal{W}\text{ satisfying \eqref{noisebound}} \}.
\end{aligned}
\end{equation}
It follows from \eqref{realdataeq} that $(A_{\textnormal{true}}, B_{\textnormal{true}}) \in \Sigma_d$, but in general, $\Sigma_d$ may include other systems. 

In the definition of $\Sigma_d$, the entries of $A_{\textnormal{true}}$ and $B_{\textnormal{true}}$ are not assumed to be known. However, in physical systems, certain system parameters are often known. For example, in a mechanical system, the variable $x_1$ may represent the position of a mass while $x_2$ represents its velocity. In this case, the first rows of the matrices $A_{\textnormal{true}}$ and $B_{\textnormal{true}}$ are fully known. We will further illustrate this in Section \ref{section Illustrative examples} where we will consider an example of a physical system. Accordingly, we will incorporate this prior knowledge about the entries of $A_{\textnormal{true}}$ and $B_{\textnormal{true}}$ into our analysis. In order to formalize the prior knowledge, we define the set $S \subseteq \{1,2,\dots, n\} \times \{1,2,\dots, f+g\}$, and assume the entry $\begin{bmatrix}
A_{\textnormal{true}} & B_{\textnormal{true}}
\end{bmatrix}_{ij}$ is given for all $(i,j) \in S$. The set of systems compatible with the prior knowledge is given by
\begin{equation} \label{set sigma p}
\begin{aligned}
    \Sigma_{pk}(S) := &\Big\{(A,B) \mid \\
    &\begin{bmatrix}
		A & B
	\end{bmatrix}_{ij} = \begin{bmatrix}
		A_{\textnormal{true}} & B_{\textnormal{true}}
	\end{bmatrix}_{ij}  \forall (i,j) \in S \Big\}.
\end{aligned}
\end{equation}
Subsequently, we define the set of systems compatible with both the data and the prior knowledge as
\begin{equation} \label{definition Sigma}
	\Sigma := \Sigma_d \cap \Sigma_{pk}(S).
\end{equation}
Note that the case that all entries of $A_{\textnormal{true}}$ and $B_{\textnormal{true}}$ are unknown can be captured by setting $S = \varnothing$, which implies $\Sigma = \Sigma_d$. It is clear from \eqref{realdataeq} and \eqref{set sigma p} that the system $(A_{\textnormal{true}},B_{\textnormal{true}})$ belongs to $\Sigma$. However, in general, $\Sigma$ contains other systems because the data may not uniquely determine $A_{\textnormal{true}}$ and $B_{\textnormal{true}}$, even if some entries of $A_{\textnormal{true}}$ and $B_{\textnormal{true}}$ are known.

The goal of this paper is to find a controller that stabilizes the origin of the system $(A_{\textnormal{true}},B_{\textnormal{true}})$. Since on the basis of the data and the prior knowledge we cannot distinguish between $(A_{\textnormal{true}}, B_{\textnormal{true}})$ and any other system in $\Sigma$, we need to find a single controller that stabilizes the origin of all systems in $\Sigma$. This motivates the following definition of informative data for stabilization of polynomial systems. In the rest of the paper, we assume that 
$$
F(0) = 0.
$$
\begin{definition} \label{definition 1}
	{\rm The data $(\dot{\mathcal{X}},\mathcal{X},\mathcal{U})$ are called} informative for stabilization {\rm if there exist a radially unbounded function $V \in \mathcal{V}$ and a continuous controller $K: \mathbb{R}^n \rightarrow \mathbb{R}^m$ such that $K(0) = 0$ and
	\begin{equation} \label{oridotVx}
		\frac{\partial V(x)}{\partial x}(A F(x) + B G(x)K(x)) < 0\quad \forall x\in \mathbb{R}^n \setminus\{ 0\},
	\end{equation}
	for all $(A,B) \in \Sigma$.}
\end{definition}
Note that for a controller $K$ satisfying $K(0) = 0$, the origin of the closed-loop system
\begin{equation} \label{closedloopsystem}
	\dot{x} = A F(x) + B G(x)K(x),
\end{equation}
is an equilibrium point, as $F(0) = 0$. If \eqref{oridotVx} holds then the origin is globally asymptotically stable for all closed-loop systems obtained by interconnecting any system $(A,B) \in \Sigma$ with the controller $u=K(x)$.

In this paper, we study the following two problems. 
\begin{problem}[Informativity] \label{problem 1}
	Find conditions under which the data $(\dot{\mathcal{X}},\mathcal{X},\mathcal{U})$ are informative for stabilization. 
\end{problem}
\begin{problem}[Controller design]  \label{problem 2}
	Suppose the data $(\dot{\mathcal{X}},\mathcal{X},\mathcal{U})$ are informative for stabilization. Find a controller $u=K(x)$ satisfying $K(0) = 0$ and \eqref{oridotVx}.
\end{problem}

\section{Connection to previous work} \label{section Connection to previous work}
Current approaches for data-driven control of polynomial systems \cite{DDCmiddleCDC,DDCmiddleTAC} build on the model-based method proposed in \cite{Originalmiddle}. These methods do not incorporate prior knowledge and instead focus on designing a common stabilizing controller for all systems compatible with the data. In these works, the controller is considered to be of the form
$$
	K(x) = Y(x)PZ(x),
$$
where $Y \in \mathbb{R}^{m \times p}[x]$, $P \in \mathbb{S}^p$ is positive definite, and $Z \in \mathbb{R}^p[x]$ is radially unbounded satisfying 
\begin{equation} \label{FHZ}
	F(x) = H(x) Z(x),
\end{equation}
for some $H \in \mathbb{R}^{f \times p}[x]$. The choice of candidate Lyapunov function 
\begin{equation} \label{V=ZPZ}
	V(x) = Z^{\top}(x) P Z(x),
\end{equation}
then leads to
$$
	\frac{\partial V}{\partial x}(x)(AF(x) + BG(x)K(x)) = 2Z^{\top}(x) P \Theta(x) PZ(x),
$$
where 
$$
	\Theta(x) := \frac{\partial Z}{\partial x}(x)\begin{bmatrix}
		A & B
	\end{bmatrix} \begin{bmatrix}
		H(x) P^{-1} \\ G(x)Y(x)
	\end{bmatrix}.
$$
The main idea in this line of work is to find $P$ and $Y(x)$ such that
\begin{equation} \label{cond middle}
	-\Theta(x)-\Theta^{\top}(x) > 0 \quad \forall x \in \mathbb{R}^n \setminus \{0\},
\end{equation}  
for all systems $(A,B)$ compatible with the data. In the earlier work \cite{DDCmiddleCDC}, $H(x)$ is taken to be equal to the identity matrix, which implies that $Z(x)=F(x)$. In contrast, \cite{DDCmiddleTAC} considers more general $Z(x)$ satisfying \eqref{FHZ}. This strategy is appealing because it leads to data-based linear matrix inequalities for control design. Unfortunately, however, the method also has some major limitations.
\begin{enumerate}
\item \label{limitation 1} The matrix $\frac{\partial Z}{\partial x}(x)$ must have full row rank for all $x \in \mathbb{R}^n \setminus \{0\}$.

Indeed, suppose that there exists a nonzero $x$ such that $\frac{\partial Z}{\partial x}(x)$ does not have full row rank. Then $\Theta(x)$ is singular, which implies that \eqref{cond middle} does not hold. Note that the full row rank condition can only hold if $p \leq n$, i.e., the number of polynomials in $Z$ is less than or equal to the state-space dimension of the system. This limits the class of Lyapunov functions of the form \eqref{V=ZPZ} that can be considered by the approach. In \cite{DDCmiddleTAC}, a strategy for choosing $Z$ is proposed, where the first $n$ components of $Z(x)$ coincide with the state $x$. However, since $\frac{\partial Z}{\partial x}(x)$ is required to have full row rank, this leaves only one possibility for $Z$, namely $Z(x) = x$. Furthermore, even if $p \leq n$, the full row rank condition rules out certain natural choices of $Z$. In fact, for relatively simple polynomial systems, the method does not yield stabilizing controllers.

For example, recall the feedback linearizable system studied in \cite{DDCmiddleCDC}:
\begin{equation} \label{example in CDC}
\begin{aligned}
\dot{x}_1 = x_2,\quad 
\dot{x}_2 = x_1^2+u.
\end{aligned}
\end{equation}
Since the paper \cite{DDCmiddleCDC} deals with the special case that $H(x) = I$ and $Z(x) = F(x)$, the authors chose $
Z(x)=\begin{bmatrix}
    x_2 & x_1^2
\end{bmatrix}^{\top}.
$
The matrix,
\begin{equation} \label{pZpx}
    \frac{\partial Z}{\partial x}(x) = \begin{bmatrix}
			0&1\\2x_1&0
    \end{bmatrix},
\end{equation}
does not have full rank when $x_1 = 0$, and therefore \eqref{cond middle} does not hold. Hence, it is not possible to find a stabilizing controller using this method, in contrast to what was claimed in  \cite[Sec. IV]{DDCmiddleCDC}.
	
Despite the increased flexibility in selecting $H(x)$ in \cite{DDCmiddleTAC}, the approach can still fail for any choice of $H(x)$. To illustrate this, consider the system
\begin{equation}\label{example moti sim}
\begin{aligned}
\dot{x}_1 = x_2-x_1^2,\quad
\dot{x}_2 = u.
\end{aligned}
\end{equation}
After defining $F(x) = \begin{bmatrix}
    x_2 & x_1^2 
\end{bmatrix}^{\top}$ and $G(x) = 1$, the matrices $A_{\textnormal{true}}$ and $B_{\textnormal{true}}$ are given by
$$
A_{\textnormal{true}} = \begin{bmatrix}
1 & -1 \\
0 & 0
\end{bmatrix}\  \text{and}\  
B_{\textnormal{true}}	 = \begin{bmatrix}
0\\
1
\end{bmatrix}.
$$
If $Z(x) = F(x)$, then $\frac{\partial Z}{\partial x}(x)$ is given by \eqref{pZpx}, which does not have full rank for all $x \in \mathbb{R}^n \setminus \{0\}$. Therefore, we conclude that \eqref{cond middle} does not hold for any choice of $H(x)$.
	
\item The matrix inequality \eqref{cond middle} is a sufficient condition for the inequality \eqref{oridotVx}, but a conservative one. \label{limitation 2}
	
Indeed, in general, \eqref{oridotVx} does not imply \eqref{cond middle}. As we demonstrate next, this can lead to failure of the approach to yield stabilizing controllers for any choice of $H(x)$ and natural choices of $Z(x)$. Consider the system in \eqref{example moti sim}. We now choose $Z(x)= x$. Clearly, all $H(x)$ are of the form
$$
H(x) = H_1(x)+J(x)\begin{bmatrix}
x_2 & -x_1
\end{bmatrix},
$$
where $J \in \mathbb{R}^2[x]$ is arbitrary and $H_1(x) = \begin{bmatrix}
0 & 1 \\ x_1 & 0
\end{bmatrix}$. Next, we partition
$$ 
\begin{aligned}
P^{-1}\! = \begin{bmatrix}
p & q\\
q & r
\end{bmatrix}\!, \   J(x) = \begin{bmatrix}
J_1(x) \\ J_2(x)
\end{bmatrix}\!, \ Y(x) = \begin{bmatrix}
Y_1(x) \\ Y_2(x)
\end{bmatrix}^{\top}\!. 
\end{aligned}
$$ 
If \eqref{cond middle} holds for all systems compatible with the data, this condition must be satisfied for the system $(A_{\textnormal{true}}, B_{\textnormal{true}})$. However, when $A = A_{\textnormal{true}}$ and $B=B_{\textnormal{true}}$, we have
$$
\Theta_{11}(x) = q-px_1+(J_1(x)-J_2(x))(px_2-qx_1).
$$
Since $P>0$, it follows that $p>0$. We now choose $x_2 = \frac{q}{p}x_1$, which leads to $\Theta_{11}(x) = q-px_1$. Therefore, there exists a sufficiently large $x_1$ such that $\Theta_{11}(x) < 0$. We conclude that \eqref{cond middle} does not hold for any choice of $H(x)$. Nonetheless, there exists a stabilizing controller for the system \eqref{example moti sim}. In fact, the controller $K(x) = -2x_1-2x_2+2x_1^2+2x_1x_2-2x_1^3$ renders the origin of the closed-loop system \eqref{closedloopsystem} globally asymptotically stable. This can be verified by taking the Lyapunov function 
$$
V(x) = x_1^2+(x_1+x_2-x_1^2)^2.
$$
\end{enumerate}

This motivates us to develop methods for data-driven stabilization of polynomial systems that do not aim for the positive definiteness of $-\Theta(x)-\Theta^{\top}(x)$ for all nonzero $x$, but that work directly with $V(x)$ and $\dot{V}(x)$. As we will show later on, the methodology developed in this paper is applicable to the systems \eqref{example in CDC} and \eqref{example moti sim}.

\section{Methodology}\label{section Methodology}

In this section, we will introduce our approach. First, in Section \ref{section Alternative description of the set}, we find an alternative expression for the set $\Sigma$ in \eqref{definition Sigma}. Subsequently, we present a data-driven control design method in Section \ref{section Data-driven control design}. This method only guarantees that \textit{almost all} trajectories of the closed-loop system converge to zero. In Section \ref{section Stability verification} we provide necessary and sufficient conditions under which the zero equilibrium point of the closed-loop system is globally asymptotically stable. Finally, computational aspects are discussed in Section \ref{section Computational approaches}.

\subsection{Alternative description of the set $\Sigma$} \label{section Alternative description of the set}
Recall that $(A,B) \in \Sigma$ if and only if $(A,B) \in \Sigma_{pk}(S)$ and 
\begin{equation} \label{dataeq}
	\dot{\mathcal{X}} = A \mathcal{F} + B \mathcal{G} \mathcal{U} + \mathcal{W}
\end{equation}
holds for some $\mathcal{W}$ satisfying the quadratic inequality \eqref{noisebound}. By vectorizing \eqref{dataeq}, we obtain
\begin{equation} \label{vec W}
	\text{vec}(\dot{\mathcal{X}}^{\top}) = \mathcal{D}^{\top} v+ \text{vec}(\mathcal{W}^{\top}),
\end{equation}
where 
\begin{equation} \label{Def D v}
	\mathcal{D} := \begin{bmatrix}
		I \otimes \mathcal{F} \\ I \otimes \mathcal{G}\mathcal{U}
	\end{bmatrix} \  \text{and}\  v := \begin{bmatrix}
		\text{vec}(A^{\top}) \\ \text{vec}(B^{\top})
	\end{bmatrix}.
\end{equation}
Let $s$ be the set of indices for which the corresponding entries of $v$ are known and let $\bar{s} = \{1,2,\dots, n(f+g)\} \setminus s$. We rewrite \eqref{vec W} as
\begin{equation} \label{vec W tilde v}
	\text{vec}(\mathcal{W}^{\top}) = \text{vec}(\dot{\mathcal{X}}^{\top}) - \mathcal{D}_{s \bullet}^{\top} v_{s} - \mathcal{D}_{\bar{s} \bullet}^{\top} v_{\bar{s}}.
\end{equation}
By the noise model \eqref{noisebound}, this implies that $(A,B) \in \Sigma$ if and only if the corresponding $v_{\bar{s}}$ satisfies
\begin{equation} \label{bound vsbar}
	\begin{bmatrix}
		1 \\ v_{\bar{s}}
	\end{bmatrix}^{\top} N
	\begin{bmatrix}
		1 \\ v_{\bar{s}}
	\end{bmatrix} \geq 0,
\end{equation}
where
\begin{equation} \label{data matrix N}
\begin{aligned}
    N := \begin{bmatrix}
1 & 0 \\
\text{vec}(\dot{\mathcal{X}}^{\top}) - \mathcal{D}_{s \bullet}^{\top} v_{s} & 	- \mathcal{D}_{\bar{s} \bullet}^{\top}
\end{bmatrix}^{\top}\Phi \quad \quad\quad\quad\quad&\\
\begin{bmatrix}
1 & 0 \\
\text{vec}(\dot{\mathcal{X}}^{\top}) - \mathcal{D}_{s \bullet}^{\top} v_{s} & 	- \mathcal{D}_{\bar{s} \bullet}^{\top}
\end{bmatrix}&.
\end{aligned}
\end{equation}
Let $\ell$ denote the cardinality of the set $\bar{s}$. Then $N \in \mathbb{S}^{1+\ell}$. With this notation in place, the set $\Sigma$, defined as in \eqref{definition Sigma}, is characterized by a quadratic inequality in terms of the corresponding $v_{\bar{s}}$. 

\subsection{Data-driven control design} \label{section Data-driven control design}
To present our data-driven control approach, we first recall the dual Lyapunov theorem proposed in \cite{DensityDual}. In the following theorem, the divergence of a differentiable function $f: \mathbb{R}^p \rightarrow \mathbb{R}^p$ is defined as
$$
	(\nabla \cdot f)(x) = \frac{\partial f_1}{\partial x_1}(x) + \frac{\partial f_2}{\partial x_2}(x) + \cdots + \frac{\partial f_p}{\partial x_p}(x).
$$
Moreover, the terminology ``almost all $x$" means all $x$ except for a set of measure zero. 
\begin{theorem}\label{theorem density}
Consider the system 
\begin{equation} \label{theorem density system}
		\dot{x}(t) = f(x(t)),
\end{equation}
where $f : \mathbb{R}^n \rightarrow \mathbb{R}^n$ is continuously differentiable and $f(0) = 0$. Suppose that there exists a $\rho : \mathbb{R}^n \rightarrow \mathbb{R}_+$ such that
\begin{enumerate}[label=(\roman*)]
\item \label{theorem density a}$\rho$ is continuously differentiable on $\mathbb{R}^n \setminus \{ 0 \}$,
\item \label{theorem density b} $\rho(x)f(x)/\|x\|$ is integrable on $\{ x \in \mathbb{R}^n \mid \|x\| \geq 1 \}$,
\item \label{theorem density c} $ \nabla \cdot (\rho f)(x) > 0$ for almost all $ x \in \mathbb{R}^n$.
\end{enumerate}
Then, for almost all initial states $x(0) \in \mathbb{R}^n$ the trajectory $x(t)$ of system \eqref{theorem density system} exists for $t \in \left[\left.0, \infty \right)\right.$ and tends to zero as $t \rightarrow \infty$.
\end{theorem}
A function $\rho$ satisfying conditions \ref{theorem density a}, \ref{theorem density b} and \ref{theorem density c} is called a \emph{density function}. We will employ Theorem \ref{theorem density} to design a data-driven stabilizing controller. To do so, we first consider, inspired by \cite{DensityConvex}, a rational density function $\rho$, say
\begin{equation} \label{rhox}
	\rho (x) = \frac{a(x)}{b^{\alpha}(x)},
\end{equation}
where $a$, $b \in \mathbb{R}[x]$ are such that $a(x) > 0$ for all $x \in \mathbb{R}^n$ and $b(x) > 0$ for all $x \in \mathbb{R}^n \setminus \{0\}$, and $\alpha > 0$. It is then clear that $\rho$ is continuously differentiable on $\mathbb{R}^n \setminus \{0\}$, so Theorem \ref{theorem density}, item \ref{theorem density a} holds. Following \cite{DensityConvex}, we will consider a controller of the form
\begin{equation} \label{rational controller}
	K(x) = \frac{c(x)}{a(x)},
\end{equation}
where $c \in \mathbb{R}^m[x]$ is such that $c(0) = 0$. We now investigate conditions \ref{theorem density b} and \ref{theorem density c} of Theorem \ref{theorem density} in the context of data-driven stabilization. In view of the closed-loop system \eqref{closedloopsystem}, we will work with $f(x) = AF(x) + BG(x)K(x)$. Therefore, with $\rho$ in \eqref{rhox} and $K$ in \eqref{rational controller}, the function $\rho(x)f(x)/\|x\|$ in Theorem \ref{theorem density}, item \ref{theorem density b} equals
\begin{equation}\label{rhofx}
	\frac{a(x) AF(x)}{b^{\alpha}(x)\|x\|} + \frac{ BG(x)c(x)}{b^{\alpha}(x)\|x\|}.
\end{equation}
Moreover, the inequality in Theorem \ref{theorem density}, item \ref{theorem density c} boils down to
\begin{equation} \label{dual Lyapunov before vec}
	\begin{aligned}
		& b(x) \nabla \cdot \left( a(x)AF(x)+BG(x)c(x) \right)\\
		& -\alpha \frac{\partial b}{\partial x} (x) \left( a(x)AF(x)+BG(x)c(x) \right)>0.
	\end{aligned}
\end{equation}
Our goal is to find $a(x)$, $b(x)$, $c(x)$ and $\alpha$ such that the following two conditions hold for all $(A,B) \in \Sigma$:
\begin{itemize}
\item the vector \eqref{rhofx} is integrable on $\{ x \in \mathbb{R}^n \mid \|x\| \geq 1 \}$,
\item the inequality \eqref{dual Lyapunov before vec} holds for almost all $ x \in \mathbb{R}^n$.
\end{itemize}
Let $A_i$ and $B_i$ be the $i$th row of $A$ and $B$, respectively, where $i = 1,2,\dots,n$. Thus, $\text{vec}(A^{\top}) = \begin{bmatrix}
	A_1 & A_2 & \cdots & A_n
\end{bmatrix}^{\top}$, and $
\text{vec}(B^{\top}) = \begin{bmatrix}
	B_1 & B_2 & \cdots & B_n
\end{bmatrix}^{\top}$. The inequality \eqref{dual Lyapunov before vec} can be written as
$$
\begin{aligned}
&b(x) \sum_{i=1}^{n}\left( \frac{ \partial(a(x)A_i F(x)+B_i G(x) c(x))}{\partial x_i}\right)\\ 
&- \alpha \sum_{i=1}^{n}\left( \frac{\partial (b(x))}{\partial x_i} (a(x)A_i F(x)+B_i G(x) c(x)) \right) >0.
\end{aligned}
$$
Equivalently,
\begin{equation} \label{Rxv}
	R^{\top}(x) v >0,
\end{equation}
where $v$ is defined as in \eqref{Def D v}, and
\begin{equation} \label{Rx}
		R(x) := \begin{bmatrix}
			b(x) \frac{ \partial(a(x)F(x))}{\partial x_1} - \alpha \frac{\partial (b(x))}{\partial x_1} a(x)F(x)\\
			\vdots\\
			b(x) \frac{ \partial(a(x)F(x))}{\partial x_n} - \alpha \frac{\partial (b(x))}{\partial x_n} a(x)F(x)\\
			b(x) \frac{ \partial(G(x)c(x))}{\partial x_1} - \alpha \frac{\partial (b(x))}{\partial x_1} G(x)c(x) \\
			\vdots \\
			b(x) \frac{ \partial(G(x)c(x))}{\partial x_n} - \alpha \frac{\partial (b(x))}{\partial x_n} G(x)c(x)
		\end{bmatrix}.
\end{equation}
Based on the definition of the index sets $s$ and $\bar{s}$, we rewrite \eqref{Rxv} as
\begin{equation} \label{Rxv tau}
	R_{s}^{\top}(x) v_{s} + R_{\bar{s}}^{\top}(x) v_{\bar{s}} > 0.
\end{equation}
Now, we introduce a data-driven control design method based on Theorem \ref{theorem density}. To this end, we first partition
$$
N = \begin{bmatrix}
	N_{11} & N_{12} \\
	N_{21} & N_{22}
\end{bmatrix},
$$
where $N_{11} \in \mathbb{R}$. Then, we will make the following blanket assumption on the data and the polynomial matrices $F$ and $G$.
\begin{assumption} \label{assumption full row rank}
	The matrix $\begin{bmatrix}
		\mathcal{F} \\ \mathcal{G}\mathcal{U}
	\end{bmatrix}$ has full row rank.
\end{assumption}
With this assumption in place, the data matrix $\mathcal{D}$, defined as in \eqref{Def D v}, also has full row rank. This implies that $N_{22}<0$. Since there exists $v_{\bar{s}}$ satisfying the inequality \eqref{bound vsbar}, it follows from \cite{DDCQMI} that $N | N_{22} \geq 0$. Building on this, we now state the main result of this subsection.
\begin{theorem} \label{theorem control design}
Suppose that there exist $a,b,\beta \in \mathbb{R}[x]$, $c \in \mathbb{R}^m[x]$ and $\alpha,\epsilon > 0$ such that
\begin{enumerate}[label=(\roman*)]
\item \label{theorem control design i} $b(x),\beta(x)> 0$ for all $x \in \mathbb{R}^n \setminus \{ 0 \}$,
\item \label{theorem control design ii} the vectors $\frac{a(x)F(x)}{b^{\alpha}(x)\|x\|}$ and $\frac{G(x)c(x)}{b^{\alpha}(x)\|x\|}$ are integrable on $\{ x \in \mathbb{R}^n \mid \|x\| \geq 1 \}$,
\item \label{theorem control design iii} $c(0) = 0$,
\item \label{theorem control design iv} $a(x) - \epsilon   \in \sos[x]$,
\item \label{theorem control design v} the SOS constraint \eqref{theorem control design v SOS} holds, where $R(x)$ is defined as in \eqref{Rx}.
\end{enumerate}
\begin{figure*}[!t]
\normalsize
\begin{equation}\label{theorem control design v SOS}
\begin{bmatrix}
-R_{\bar{s}}^{\top}(x) N_{22}^{-1}N_{21} + R_{s}^{\top}(x) v_{s} -\beta(x) & (N|N_{22})^{\frac{1}{2}}R_{\bar{s}}^{\top}(x) \\
(N|N_{22})^{\frac{1}{2}}R_{\bar{s}}(x) & \left(R_{\bar{s}}^{\top}(x) N_{22}^{-1}N_{21} - R_{s}^{\top}(x) v_{s} \right) N_{22}
\end{bmatrix}\in \sos^{1+\ell}[x]
\end{equation}
\hrulefill
\end{figure*}
Let $K(x) = \frac{c(x)}{a(x)}$. Then, for any $(A,B) \in \Sigma$ and almost all initial conditions $x(0)\in \mathbb{R}^n$, the state trajectory $x(t)$ of the closed-loop system \eqref{closedloopsystem} tends to $0$ as $t \rightarrow \infty$.
\end{theorem}
The proof of this theorem is given in Section~\ref{section Proof of Theorem 2}. Two noteworthy remarks are in order. The first remark clarifies a subtle point about stability, while the second one provides computational guidelines.
\begin{remark}
Theorem \ref{theorem control design} cannot guarantee \emph{asymptotic stability} of \eqref{closedloopsystem} for all $(A,B) \in \Sigma$. Instead, it guarantees that the trajectories of the closed-loop system, obtained from interconnecting any system in $\Sigma$ with the controller, converge to zero for \emph{almost all initial states}.
\end{remark}
\begin{remark}
Note that the conditions stated in Theorem \ref{theorem control design} are linear in $a(x)$ and $c(x)$, once the polynomials $b(x)$, $\beta(x)$ and the constants $\epsilon$, $\alpha$ are fixed. Based on this observation, one can derive a controller from Theorem \ref{theorem control design} by the following steps:
\begin{enumerate}
\item \label{step 1} Choose $b, \beta\in \mathbb{R}[x]$ such that Theorem \ref{theorem control design}, item \ref{theorem control design i} holds.
\item \label{step 2} Fix the degrees of $a(x)$ and $c(x)$.
\item \label{step 3} Choose sufficiently large $\alpha$ such that Theorem \ref{theorem control design}, item \ref{theorem control design ii} holds.
\item \label{step 4} Choose $\epsilon > 0$ and solve an SOS program subject to the constraints specified in Theorem \ref{theorem control design}, item \ref{theorem control design iii}, \ref{theorem control design iv} and \ref{theorem control design v} to compute $a \in \mathbb{R}[x]$ and $c \in \mathbb{R}^m[x]$.
\end{enumerate}
Following the clarification of the steps for obtaining a controller using Theorem \ref{theorem control design}, we now discuss considerations that guide the choices required in steps (1)-(4). In particular, the choice of $b(x)$ may require case-specific insight. We refer to examples in \cite{DensityConvex}, where various strategies are proposed in the model-based setting that offer guidance for selecting $b(x)$ in the data-driven setting. The function $\beta(x)$ can be chosen to be of the form $\hat{\beta}\|x\|^{2p}$, where $p$ is a positive integer and the constant $\hat{\beta}>0$ is sufficiently small. Additionally, choosing higher degrees for $a(x)$ and $c(x)$ enlarges the set of allowed controllers with the price of increased computational complexity. To satisfy item \ref{theorem control design iv} of Theorem \ref{theorem control design}, the degree of $a(x)$ should be even. 
\end{remark}
In the next subsection, we present methods to verify asymptotic stability of the closed-loop system for all systems compatible with the data and the prior knowledge, given the controller proposed in Theorem \ref{theorem control design}.

\subsection{Stability verification} \label{section Stability verification}
In this section, we will provide conditions under which the origin of \eqref{closedloopsystem} is asymptotically stable for all $(A,B) \in \Sigma$, given the controller $K(x)$. We start with the following definition of informativity for closed-loop stability.
\begin{definition} \label{definition 2}
    {\rm Let $K: \mathbb{R}^n \rightarrow \mathbb{R}^m$ be a continuous function such that $K(0) = 0$. The data $(\dot{\mathcal{X}},\mathcal{X},\mathcal{U})$ are called} informative for closed-loop stability {\rm with respect to $K$ if there exists a radially unbounded $V \in \mathcal{V}$ such that \eqref{oridotVx} holds for all $(A,B) \in \Sigma$.}
\end{definition}
It is clear that if the data $(\dot{\mathcal{X}}, \mathcal{X}, \mathcal{U})$ are informative for closed-loop stability with respect to $K$, they are informative for stabilization. The idea is to construct a candidate controller $K$ using Theorem \ref{theorem control design}, after which we will verify informativity for closed-loop stability with respect to this $K$. To introduce our verification method, we first reformulate the inequality \eqref{oridotVx} in terms of $v_{\bar{s}}$. By vectorizing $A$ and $B$, we rewrite \eqref{oridotVx} as
\begin{equation}
	L^{\top}(x) v >0\quad \forall x\in \mathbb{R}^n \setminus\{ 0\},
\end{equation}
where $v$ is defined as in \eqref{Def D v}, and
\begin{equation} \label{Lx}
	L(x) := \begin{bmatrix}
		-I \otimes F(x) \\ -I \otimes G(x)K(x)
	\end{bmatrix}\frac{\partial V}{\partial x}^{\top}(x).
\end{equation}
Building on the definition of $s$ and $\bar{s}$, we have
\begin{equation} \label{Lxv tau}
	L_{s}^{\top}(x) v_{s} + L_{\bar{s}}^{\top}(x) v_{\bar{s}} > 0 \quad \forall x\in \mathbb{R}^n \setminus\{ 0\}.
\end{equation}
Then, investigating informativity for closed-loop stability boils down to providing conditions under which all vectors $v_{\bar{s}}$ satisfying \eqref{bound vsbar} also satisfy \eqref{Lxv tau}. Based on this, we now present necessary and sufficient conditions under which the data are informative for closed-loop stability.

\begin{theorem} \label{theorem DDC strict}
    Let $K: \mathbb{R}^n \rightarrow \mathbb{R}^m$ be a continuous function such that $K(0) = 0$. The data $(\dot{\mathcal{X}},\mathcal{X},\mathcal{U})$ are informative for closed-loop stability if and only if there exist a radially unbounded $V \in \mathcal{V}$ and a $\beta: \mathbb{R}^n \rightarrow \mathbb{R}$ such that for all $x\in \mathbb{R}^n \setminus\{ 0\}$, $\beta(x) > 0$, and \eqref{theorem DDC strict b matrix inequality} holds, where $L(x)$ is defined as in \eqref{Lx}.
    \begin{figure*}[!t]
    \normalsize
    \begin{equation} \label{theorem DDC strict b matrix inequality}
            \begin{bmatrix}
                -L_{\bar{s}}^{\top}(x)N_{22}^{-1}N_{21} + L_{s}^{\top}(x) v_{s} - \beta(x) & (N|N_{22})^{\frac{1}{2}}L_{\bar{s}}^{\top}(x) \\
				(N|N_{22})^{\frac{1}{2}}L_{\bar{s}}(x) & \left(L_{\bar{s}}^{\top}(x)N_{22}^{-1}N_{21} - L_{s}^{\top}(x) v_{s}\right)N_{22}
		\end{bmatrix} \geq 0
	\end{equation}
    \hrulefill
    \end{figure*}
    
    If the latter conditions hold, then the origin of the closed-loop system $\dot{x} = AF(x) + BG(x)K(x)$ is globally asymptotically stable for all $(A,B) \in \Sigma$. Moreover, the data $(\dot{\mathcal{X}},\mathcal{X},\mathcal{U})$ are informative for stabilization.
\end{theorem}
The proof of Theorem \ref{theorem DDC strict} is given in Section~\ref{section Proof of Theorem 3}. This relies on a specialized S-lemma that provides conditions under which a quadratic inequality implies a linear one. Equivalent conditions for informativity for closed-loop stability can be obtained by applying the classical S-lemma. For the sake of completeness, we will also state such conditions in Proposition \ref{proposition comparison}. However, as we will explain shortly, Proposition \ref{proposition comparison} suffers from the drawback of introducing an additional state-dependent multiplier $\gamma(x)$ which is not present in Theorem \ref{theorem DDC strict}.

\begin{proposition} \label{proposition comparison}
    Let $K: \mathbb{R}^n \rightarrow \mathbb{R}^m$ be continuous such that $K(0) = 0$. The data $(\dot{\mathcal{X}},\mathcal{X},\mathcal{U})$ are informative for closed-loop stability if and only if there exist a radially unbounded $V \in \mathcal{V}$ and functions $\gamma,\ \beta: \mathbb{R}^n \rightarrow \mathbb{R}$ such that for all $x \in \mathbb{R}^n \setminus \{ 0\}$, $\gamma(x) \geq 0$, $\beta (x) > 0$ and
    \begin{equation} \label{theorem DDC strict b1 matrix inequality}
	\begin{bmatrix}
            2L_{s}^{\top}(x) v_{s} & L_{\bar{s}}^{\top}(x) \\ L_{\bar{s}}(x) & 0
	\end{bmatrix} - \gamma (x) N \geq
        \begin{bmatrix}
		\beta(x)  & 0 \\ 0 & 0
	\end{bmatrix},
    \end{equation}
    where $L(x)$ is defined as in \eqref{Lx}.
\end{proposition}

The proof of Proposition \ref{proposition comparison} is given in Section \ref{section Proof of Proposition 4}. Compared to the equivalent conditions presented in Theorem \ref{theorem DDC strict}, the main difference in Proposition \ref{proposition comparison} is the introduction of an additional state-dependent multiplier $\gamma(x)$. This is important from a computational point of view. In fact, in order to verify the inequalities \eqref{theorem DDC strict b matrix inequality} and \eqref{theorem DDC strict b1 matrix inequality}, we rely on SOS programming. To this end, all entries of the involved matrices are required to be polynomial. In particular, $\gamma(x)$ has to be parameterized as a polynomial, which introduces additional coefficients to be determined and increases the computational complexity. However, Theorem \ref{theorem DDC strict} bypasses the parameterization of the multiplier, demonstrating that this step is not necessary. This leads to a clear computational benefit. For this reason, our method avoids introducing a multiplier $\gamma(x)$ and works with Theorem \ref{theorem DDC strict} instead.

\subsection{Computational approaches} \label{section Computational approaches}
In the previous subsection, we have provided conditions under which the data are informative for closed-loop stability. These conditions involve checking the inequality \eqref{theorem DDC strict b matrix inequality} for all $x \in \mathbb{R}^n\setminus\{0\}$. In this subsection, we present a more computationally tractable SOS program that implies the feasibility of \eqref{theorem DDC strict b matrix inequality}. To cover the controllers obtained using Theorem \ref{theorem control design}, throughout this subsection, we focus on rational controllers with a positive denominator. In particular, let $K(x) = \frac{c(x)}{a(x)}$, where $a \in \mathbb{R}[x]$ and $c \in \mathbb{R}^m [x]$ satisfy $a(x) > 0$ for all $x \in \mathbb{R}^n$ and $c(0) = 0$. 
\subsubsection{Application of SOS relaxation} \label{section Application of SOS relaxation}
To apply SOS relaxation, we require all entries of the involved matrix to be polynomial. However, since $K(x)$ is rational, the entries of $L(x)$ are in general not polynomial. To address this, we simply multiply the inequality \eqref{theorem DDC strict b matrix inequality} from both sides by $a(x)$. Note that since $a(x) > 0$ for all $x \in \mathbb{R}^n$, this does not change the inequality. However, it has the benefit that we can work with the vector 
$$
    a(x)L(x) = \begin{bmatrix}
        -I \otimes a(x)F(x) \\ -I \otimes G(x)c(x)
    \end{bmatrix}\frac{\partial V}{\partial x}^{\top}(x),
$$
which is, in contrast to $L(x)$, a member of $\mathbb{R}^{n(f+g)}[x]$. Next, we present a computationally tractable approach to verify data informativity for closed-loop stability.
\begin{corollary} \label{corollary global}
    Let $a \in \mathbb{R}[x]$ and $c \in \mathbb{R}^m[x]$, where $a(x) > 0$ for all $x\in \mathbb{R}^n$ and $c(0) = 0$. Consider the controller $K(x) = \frac{c(x)}{a(x)}$. The data $(\dot{\mathcal{X}},\mathcal{X},\mathcal{U})$ are informative for closed-loop stability with respect to $K$ if there exist $\epsilon,\ \beta,\ V \in \mathbb{R}[x]$ such that
    \begin{enumerate}[label=(\roman*)]
	\item \label{corollary global a} $V(0) = 0$,
        \item \label{corollary global b} $\epsilon(x) > 0$ and $ \beta(x) > 0$ for all $x\in \mathbb{R}^n \setminus\{ 0\}$, and $\epsilon(x)$ is radially unbounded,
	\item \label{corollary global c} $V(x) - \epsilon(x) \in  \sos[x]$,
	\item \label{corollary global d} the SOS constraint \eqref{corollary global d SOS} holds, where $L(x)$ is defined as in \eqref{Lx}. 
    \end{enumerate}
    \begin{figure*}[!t]
    \normalsize
    \begin{equation}\label{corollary global d SOS}
        \begin{bmatrix}
            -a(x)L_{\bar{s}}^{\top}(x)N_{22}^{-1}N_{21} + a(x)L_{s}^{\top}(x) v_{s} - \beta(x) & a(x)(N|N_{22})^{\frac{1}{2}}L_{\bar{s}}^{\top}(x) \\
            a(x)(N|N_{22})^{\frac{1}{2}}L_{\bar{s}}(x) & a(x)\left(L_{\bar{s}}^{\top}(x)N_{22}^{-1}N_{21} - L_{s}^{\top}(x) v_{s}\right)N_{22}
        \end{bmatrix}  \in  \sos^{1+\ell}[x]
    \end{equation}
    \hrulefill
    \end{figure*}
\end{corollary}
The proof of Corollary \ref{corollary global} is given in Section \ref{section Proof of Corollary global}. This corollary provides a computational approach to find a common Lyapunov function for all closed-loop systems $\dot{x} = AF(x) + BG(x)K(x)$ compatible with the data and the prior knowledge.

\subsubsection{Reduction of computational complexity} \label{section Reduction of computational complexity}
In this section, we introduce an alternative computational approach to verify closed-loop stability, resulting in matrices with reduced size. In Corollary \ref{corollary global}, item \ref{corollary global d}, the size of the matrix is $(1+ \ell)\times(1+\ell)$, which is independent of the time horizon $T$ of the experiment. However, this size can still be large, as $\ell$ represents the number of unknown entries of $A_{\textnormal{true}}$ and $B_{\textnormal{true}}$, which can be as large as $n(f+g)$. Recognizing this limitation, we propose an alternative approach, where the matrix size depends only on the state-space dimension. To achieve this, we first reformulate the matrix inequality \eqref{theorem DDC strict b matrix inequality}. Define
\begin{equation} \label{Qx}
    Q(x):=\begin{bmatrix}
		-I \otimes a(x)F(x) \\ -I \otimes G(x)c(x)
    \end{bmatrix}.
\end{equation}
Given $L(x)$, defined as in \eqref{Lx}, we have 
$$
    a(x)L_{\bar{s}}(x) =  Q_{\bar{s}\bullet}(x)\frac{\partial V}{\partial x}^{\top}(x).
$$
Since $N_{22} < 0$ and $a(x) > 0$ for all $x\in \mathbb{R}^n$, by using a Schur complement argument, the matrix inequality \eqref{theorem DDC strict b matrix inequality} is equivalent to \eqref{eq form 1}.
\begin{figure*}[!t]
    \normalsize
    \begin{equation} \label{eq form 1}
	\begin{bmatrix}
		-L_{\bar{s}}^{\top}(x)N_{22}^{-1}N_{21} + L_{s}^{\top}(x) v_{s} - \beta(x) & (N|N_{22})^{\frac{1}{2}}L_{\bar{s}}^{\top}(x)N_{22}^{-1}Q_{\bar{s} \bullet}(x) \\
		(N|N_{22})^{\frac{1}{2}}Q_{\bar{s} \bullet}^{\top}(x)N_{22}^{-1}L_{\bar{s}}(x) & \left(L_{\bar{s}}^{\top}(x)N_{22}^{-1}N_{21} - L_{s}^{\top}(x) v_{s}\right)Q_{\bar{s} \bullet}^{\top}(x)N_{22}^{-1}Q_{\bar{s} \bullet}(x)
	\end{bmatrix}
	 \geq 0.
    \end{equation}
    \hrulefill
\end{figure*}
Clearly, given a controller $K$, this matrix inequality remains linear in the decision variables. We observe that the size of the involved matrix is $(1+n) \times (1+n)$, which depends only on the state-space dimension. Building on this, we will construct new SOS constraints. Note that although the given $K$ is rational, the entries of $a(x)L(x)$ and $Q(x)$ are polynomial when $V \in \mathcal{V}$. The following corollary provides the computational approach corresponding to \eqref{eq form 1}.
\begin{figure}[t]
    \begin{center}
        \includegraphics[trim=40 30 50 60, clip, width=0.95\columnwidth]{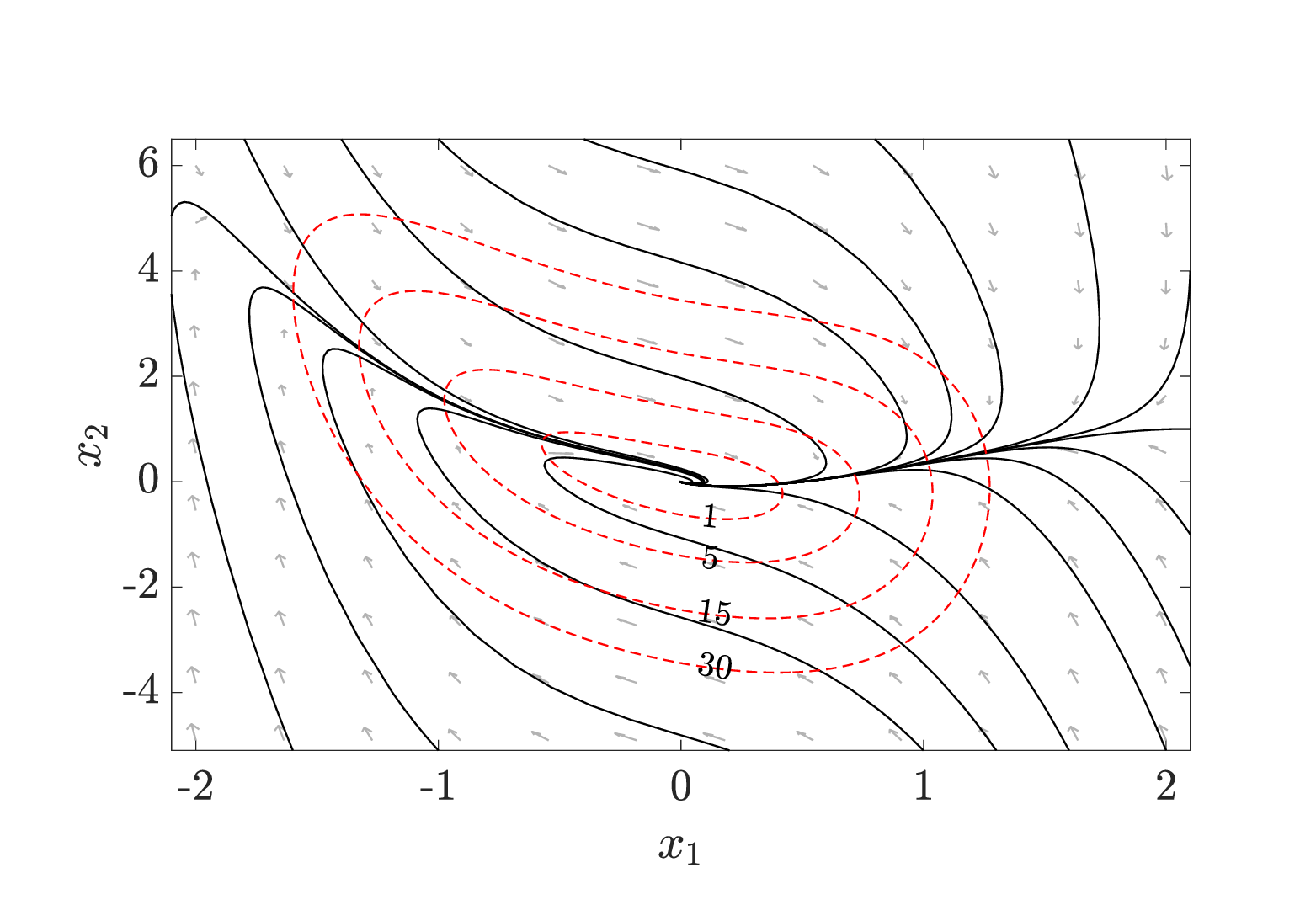}
        \caption{The phase portrait of $\dot{x}=A_{\textnormal{true}}F(x)+ B_{\textnormal{true}}G(x) K(x)$ and the level sets of $V(x)$. The gray arrows show the closed-loop vector field and the black lines are trajectories starting from the edges and converging to the origin. The level sets of the obtained $V(x)$ are represented by red dashed lines.}
	\label{fig Ex2}
    \end{center}
\end{figure}
\begin{corollary} \label{corollary alternative}
    Let $a \in \mathbb{R}[x]$ and $c \in \mathbb{R}^m[x]$, where $a(x) > 0$ for all $x\in \mathbb{R}^n$ and $c(0) = 0$. Consider the controller $K(x) = \frac{c(x)}{a(x)}$. The data $(\dot{\mathcal{X}},\mathcal{X},\mathcal{U})$ are informative for closed-loop stability with respect to $K$ if there exist $\epsilon,\ \beta,\ V \in \mathbb{R}[x]$ such that conditions \ref{corollary global a}, \ref{corollary global b} and \ref{corollary global c} in Corollary \ref{corollary global} hold, and the SOS constraint \eqref{corollary alternative d SOS} is satisfied, where $L(x)$ and $Q(x)$ are defined as in \eqref{Lx} and \eqref{Qx}, respectively. 
    \begin{figure*}[!t]
    \normalsize
    \begin{equation} \label{corollary alternative d SOS}
	\begin{bmatrix}
            -a(x)(L_{\bar{s}}^{\top}(x)N_{22}^{-1}N_{21} + L_{s}^{\top}(x) v_{s}) - \beta(x) & a(x)(N|N_{22})^{\frac{1}{2}}L_{\bar{s}}^{\top}(x)N_{22}^{-1}Q_{\bullet \bar{s}}^{\top}(x) \\
            a(x)(N|N_{22})^{\frac{1}{2}}Q_{\bullet \bar{s}}(x)N_{22}^{-1}L_{\bar{s}}(x) & a(x)\left(L_{\bar{s}}^{\top}(x)N_{22}^{-1}N_{21} - L_{s}^{\top}(x) v_{s}\right)Q_{\bullet \bar{s}}(x)N_{22}^{-1}Q_{\bullet \bar{s}}^{\top}(x)
	\end{bmatrix}  \in \sos^{1+n}[x],
    \end{equation}
    \hrulefill
    \end{figure*}
\end{corollary}
The proof of Corollary \ref{corollary alternative} is given in Section~\ref{section Proof of Corollary alternative}. Compared to Corollary \ref{corollary global}, the size of the SOS constraint in Corollary \ref{corollary alternative} depends on the state-space dimension $n$, rather than the number of unknown parameters $\ell$.

\section{Illustrative examples} \label{section Illustrative examples}
In this section, we illustrate Theorem~\ref{theorem control design}, Corollary~\ref{corollary global} and Corollary~\ref{corollary alternative} with three examples. In Example~\ref{example 1}, we revisit the system~\eqref{example in CDC} using data without noise, while in Example~\ref{example 2}, the system~\eqref{example moti sim} is studied with noisy data. Additionally, a mathematical model of a DC motor with cubic damping is examined in Example~\ref{example DC}.

\begin{example} \label{example 1}
In this example, we again consider the system \eqref{example in CDC}. After taking $F(x) = \begin{bmatrix}
x_2 & x_1^2 
\end{bmatrix}^{\top}$ and $G(x) = 1$, the matrices $A_{\textnormal{true}}$ and $B_{\textnormal{true}}$ are given by
$$ 
A_{\textnormal{true}} = \begin{bmatrix}
1 & 0 \\
0 & 1
\end{bmatrix}\quad \text{and} \quad
B_{\textnormal{true}}= \begin{bmatrix}
0\\1
\end{bmatrix}.
$$
Let $x(0) = \begin{bmatrix}
1 & -1
\end{bmatrix}^{\top}$ and $u(t)= -\sin(2t)+\cos(t)$ for $t \in [0,3]$. We collect $T=3$ data samples from the system at the time instants $t_0 = 0.4$, $t_1 = 0.6$ and $t_2 = 0.8$. During the experiment, we assume that the noise samples are zero, and therefore, $\Phi_{11}=0$. The data matrices are
$$
\begin{aligned}
	\dot{\mathcal{X}} = \begin{bmatrix}
		-0.4750 &  -0.3748 &  -0.3452\\
		0.7249 &   0.3008  &  0.0190
	\end{bmatrix}, 
\end{aligned}
$$
$$
\begin{aligned}
	\mathcal{X} &= \begin{bmatrix}
		0.7219 & 0.6384 & 0.5673\\
		-0.4750 &  -0.3748 &  -0.3452
	\end{bmatrix},\\
	\mathcal{U} &= \begin{bmatrix}
		0.2037 & 0 & 0 \\ 
        0 & -0.1067 & 0 \\ 
        0 & 0 & -0.3029
	\end{bmatrix}.
\end{aligned}
$$
Furthermore, we have
$$
\begin{aligned}
		\mathcal{F} &= \begin{bmatrix}
			-0.4750 &  -0.3748 &  -0.3452\\
			0.5212  &  0.4075 &  0.3219
		\end{bmatrix},\\
		\mathcal{GU} &= \begin{bmatrix}
			0.2037 &  -0.1067  & -0.3029
		\end{bmatrix},
\end{aligned}
$$
with which Assumption \ref{assumption full row rank} is satisfied. In this example, we assume that all entries of $A_{\textnormal{true}}$ and $B_{\textnormal{true}}$ are unknown, which implies that $ s = \varnothing$ and $\Sigma = \Sigma_d$. Further, $v_{\bar{s}} = v$ and $\mathcal{D}_{\bar{s}\bullet} = \mathcal{D}$.
	
We then take $a(x)$ to be a positive constant and $c(x)$ to be a polynomial of degree at most $2$. To ensure $c(0) = 0$, we impose the restriction that the constant term of $c(x)$ is zero. Let $b(x) = x_1^2 + (x_1+x_2)^2$, $\epsilon = 10^{-4}$ and $\beta(x) = 10^{-4}(x_1^2 + x_2^2)$. The condition \ref{theorem control design ii} in Theorem \ref{theorem control design} is satisfied by choosing $\alpha = 4$. Define $R(x)$ as in \eqref{Rx}. Finally, we formulate an SOS program that imposes the constraints in Theorem \ref{theorem control design}, items \ref{theorem control design iv} and \ref{theorem control design v}, and we solve this for $a(x)$ and $c(x)$. The simulations are conducted in MATLAB, using YALMIP \cite{Yalmip} with the SOS module \cite{YalmipSOS}, and the solver MOSEK \cite{mosek}. We obtain $a(x) = 2.3222$, $c(x) = -2.2991x_1-3.9949x_2 - 2.3222x_1^2$ and
\begin{equation} \label{example 1 Kx}
		K(x) = \frac{c(x)}{a(x)} = -0.99x_1-1.7203x_2 - x_1^2 .
\end{equation}
Subsequently, given the controller $K(x)$, we apply Corollary \ref{corollary global} to verify the global asymptotic stability of the closed-loop system for all systems compatible with the data and the prior knowledge. We take $V$ to be a polynomial of degree at most 2. To ensure $V(0) = 0$, we take the constant term of $V$ to be zero. Define $L(x)$ as in \eqref{Lx}. We choose $\epsilon(x) = \beta(x) = 10^{-4}(x_1^2+x_2^2)$. Then, we formulate an SOS program that imposes conditions \ref{corollary global c} and \ref{corollary global d} in Corollary \ref{corollary global}, and we solve this for $V$. We obtain a radially unbounded function $V \in \mathcal{V}$, given by
    $$
	V(x) = 3.0231 x_1^2 + 2.1599 x_1x_2 + 1.6122 x_2^2.
    $$ 
    This verifies that the origin of the closed-loop system, obtained by interconnecting any system compatible with the data and the controller \eqref{example 1 Kx}, is globally asymptotically stable.
\end{example}

\begin{example} \label{example 2}
In this example, we again consider the system \eqref{example moti sim}. We take $	F(x) := \begin{bmatrix}
		x_1 & x_2 & x_1^2 
\end{bmatrix}^{\top}$ and $G(x) := 1$. Then, the matrices $A_{\textnormal{true}}$ and $B_{\textnormal{true}}$ are
$$
	A_{\textnormal{true}} = \begin{bmatrix}
		0 & 1 & -1 \\
		0 & 0 & 0 \\
	\end{bmatrix}\quad \text{and}\quad B_{\textnormal{true}} = \begin{bmatrix}
		0\\1
	\end{bmatrix}.
$$
Let $x(0) = \begin{bmatrix}
		-1 & 1
\end{bmatrix}^{\top}$ and $u(t)= 10\sin(5t)+5\cos(3t)$ for $t \in [0,5]$. We collect $T=4$ data samples from the system at the time instants $t_0 = 0.40$, $t_1 = 0.47$, $t_2 = 0.54$ and $t_3 = 0.61$. 
During the experiment, the noise samples are drawn independently and uniformly at random from the ball $\{r \in \mathbb{R}^2 \mid \|r\| \leq \omega \}$, where $\omega = 0.005$. Define $\Phi_{11} = \omega^2 T $. The data matrices are
$$
	\begin{aligned}
		\dot{\mathcal{X}} &= \begin{bmatrix}
			5.3834  &  5.9280 &   5.8807 &   5.2518\\
			10.9048 &   7.9133 &   4.0277 &  -0.3664
		\end{bmatrix},\\
		\mathcal{X} &= \begin{bmatrix}
			-0.0468  &  0.3522 &   0.7691 &   1.1618\\
			5.3856  &  6.0505  &  6.4726  &  6.6024
		\end{bmatrix},\\
		\mathcal{U} &= \begin{bmatrix}
			10.9048 & 0 & 0 & 0 \\
			0 & 7.9153 & 0 & 0 \\
			0 & 0 & 4.0279 & 0 \\
			0 & 0 & 0 & -0.3669
		\end{bmatrix}.
    	\end{aligned}
$$
Furthermore, we have
$$
	\begin{aligned}
		\mathcal{F} &= \begin{bmatrix}
			-0.0468  &  0.3522 &   0.7691 &   1.1618\\
			5.3856  &  6.0505  &  6.4726  &  6.6024\\
			0.0022  &  0.1240  &  0.5915  &  1.3498
		\end{bmatrix},\\
		\mathcal{GU} &= \begin{bmatrix}
			10.9048  &  7.9153  &  4.0279  & -0.3669
		\end{bmatrix},
	\end{aligned}
$$
with which Assumption \ref{assumption full row rank} is satisfied. We assume that the second row of $A_{\textnormal{true}}$ and the first entry of $B_{\textnormal{true}}$ are known to be zero. Therefore, we have $s = \{4,5,6,7\}$ and $v_s=0$. Let $\bar{s} = \{1,2,3,8\}$.

We then take $a(x)$ to be a positive constant and $c(x)$ to be a polynomial of degree at most $5$. Moreover, we impose the restriction that the constant term of $c(x)$ is zero, ensuring $c(0) = 0$. Let $b(x) = x_1^2 + (x_1+x_2+x_1^3)^2$, $\epsilon = 10^{-4}$ and $\beta(x) = 10^{-4}(x_1^2 + x_2^2)$. The condition \ref{theorem control design ii} in Theorem \ref{theorem control design} is satisfied by choosing $\alpha = 4$. Define $R(x)$ as in \eqref{Rx}. With this notation in place, we write an SOS program imposing constraints \ref{theorem control design iv} and \ref{theorem control design v} in Theorem \ref{theorem control design}, and we solve this for $a(x)$ and $c(x)$. We obtain $a(x) = 0.7289$ and
$$
\begin{aligned}
c(x) =& -0.6434 x_1 - 1.0037 x_2 + 0.8033 x_1^2\\ 
& -0.2336 x_1x_2 -0.6258 x_1^3-2.6986 x_1^2x_2\\
& +1.9630x_1^4-0.5139 x_1^5.
\end{aligned}
$$
Since the controller $K(x) = \frac{c(x)}{a(x)}$, we have
\begin{equation} \label{example 2 Kx}
\begin{aligned}
K(x) =& -0.8828x_1-1.3770x_2 + 1.1021 x_1^2\\ 
& -0.3205 x_1x_2 -0.8585 x_1^3-3.7025x_1^2x_2\\
& +2.6932x_1^4-0.7051x_1^5.
\end{aligned}
\end{equation}
Subsequently, given this controller $K(x)$, we apply Corollary \ref{corollary alternative} to verify the global asymptotic stability of the closed-loop system for all systems compatible with the data and the prior knowledge. We take $V$ to be a polynomial of degree at most 4. To ensure $V(0) = 0$, we take the constant term of $V$ to be zero. Define $L(x)$ and $Q(x)$ as in \eqref{Lx} and \eqref{Qx}, respectively. Moreover, we choose $\epsilon(x) = \beta(x) = 10^{-4}(x_1^2+x_2^2)$. Then, we formulate an SOS program imposing item \ref{corollary global c} in Corollary \ref{corollary global} and the constraint \eqref{corollary alternative d SOS}, and we solve this for $V$. We obtain a radially unbounded function $V \in \mathcal{V}$ of degree $4$. The phase portrait of the closed-loop system $\dot{x}=A_{\textnormal{true}}F(x)+B_{\textnormal{true}}G(x)K(x)$ and the level sets of $V(x)$ are illustrated in Fig.~\ref{fig Ex2}.
		
In this example, the size of the matrix involved in \eqref{corollary alternative d SOS}, is $3 \times 3$. In comparison, if Corollary \ref{corollary global} would be applied, this would result in a matrix of size $5 \times 5$. This reduction in matrix size is a significant advantage of the alternative approach developed in Section~\ref{section Reduction of computational complexity}. 	
\end{example}
\begin{example} \label{example DC}
    Consider the system
    $$
    \begin{aligned}
        L \dot{I}(t) + R I(t) + K_{s} r\dot{\theta}(t) &= V_c(t)\\
        J r\ddot{\theta}(t) + c r \dot{\theta}^3(t) &= K_i I(t),
    \end{aligned}
    $$
    where $I(t),\theta(t),V_c(t) \in \mathbb{R}$. Here, we denote the current by $I(t)$ and the voltage by $V_c(t)$, both as functions of time $t$. The angle of the motor is $\theta(t)$ and $\dot{\theta}(t)$ is the angular velocity. Moreover, $L$ is the electric inductance and $R$ is the resistance. The motor torque constant is denoted by $K_{s}$ and the electromotive force constant is given by $K_i$. Finally,  $r$, $J$ and $c$ denote the reduction ratio of gears, the moment of inertia of the rotor and the motor viscous damping constant, respectively. We consider $V_c$ as the input and $\theta$, $\dot{\theta}$ and $I$ as the state variables. In this example, we let the true parameters take the following values:
    $$
	\begin{aligned}
		& L = 1\ \mathrm{H}, \\
            & r= 10,\\
	\end{aligned}\quad 
	\begin{aligned}
		& R = 0.5\ \mathrm{\Omega}, \\
            & J = 0.01\ \mathrm{kg/m^2 }, 
	\end{aligned}\quad
	\begin{aligned}
		& K_{s} = K_i = 0.1\ \mathrm{N \cdot m/A}, \\
            & c = 0.01\ \mathrm{N \cdot m \cdot s}. 
	\end{aligned} 
    $$
    The units of $\theta$, $\dot{\theta}$, $I$ and $t$ are $\mathrm{rad}$, $\mathrm{rad/s}$, $\mathrm{A}$ and $\mathrm{s}$, respectively. We aim to design a controller to stabilize the equilibrium point $\theta = 0$, $\dot{\theta} = 0$ and $I = 0$. After defining $x := \begin{bmatrix}
		\theta & \dot{\theta} & I
	\end{bmatrix}^{\top}$ and $u := V_c$, the state-space equations are
	$$
		\begin{aligned}
			\dot{x}_1 &= x_2, \\
			\dot{x}_2 &= x_3 - x_2^3, \\
			\dot{x}_3 &= -x_2 - 0.5x_3 + u.
		\end{aligned}
	$$
    After taking $F(x) := \begin{bmatrix}
		x_2 & x_3 & x_2^3
	\end{bmatrix}^{\top}$ and $G(x) := 1$, the matrices $A_{\textnormal{true}}$ and $B_{\textnormal{true}}$ are
    $$
    A_{\textnormal{true}} = \begin{bmatrix}
		1 & 0 & 0\\
		0 & 1 & -1\\
		-1 & -0.5 & 0
	\end{bmatrix} \quad \text{and} \quad B_{\textnormal{true}} = \begin{bmatrix}
		0\\0\\1
	\end{bmatrix}.
    $$
    Let $x(0) = \begin{bmatrix}
		-10 & 10 & 10
    \end{bmatrix}^{\top}$ and $u(t)= -30\sin(10t)$ for $t \in [0,10]$. We collect $T = 5$ data samples from the system at the time instants $t_0 = 0.25$, $t_1 = 0.30$, $t_2 = 0.35$, $t_3 = 0.40$ and $t_4 = 0.45$. During the experiment, the noise samples are drawn independently and uniformly at random from the ball $\{r \in \mathbb{R}^3 \mid \|r\| \leq \omega \}$, where $\omega = 0.01$. Define $\Phi_{11} = \omega^2 T $. The data matrices are
    $$
	\begin{aligned}
            \dot{\mathcal{X}} &= \begin{bmatrix}
		      1.8215  &  1.6843   & 1.5903  &  1.5352   & 1.5359\\
                -2.9921  & -2.4427   & -1.6013  & -0.5105  &  0.6270\\
                -21.3239  & -7.1086  &  7.7433  & 19.6249  & 25.6670
	\end{bmatrix},\\
    \mathcal{X} &= \begin{bmatrix}
                -9.3414  & -9.2538  & -9.1722   &-9.0944  & -9.0179 \\
                1.8252  &  1.6882   & 1.5858  &  1.5323  &  1.5356\\
                3.0849  &  2.3639   & 2.3850 &   3.0885 &   4.2495
		\end{bmatrix},\\
        \mathcal{U} &= \begin{bmatrix}
			-17.9542 & 0 & 0 & 0 & 0 \\
			0 & -4.2336 & 0 & 0 & 0 \\
			0 & 0 & 10.5235 & 0 & 0 \\
			0 & 0 & 0 & 22.7041 & 0 \\
                0 & 0 & 0 & 0 & 29.3259
		\end{bmatrix}.
    \end{aligned}
    $$
    Moreover, we have
    $$
	\begin{aligned}
	\mathcal{F} &= \begin{bmatrix}
		1.8252  &  1.6882 &   1.5858  &  1.5323 &   1.5356 \\
            3.0849  &  2.3639    & 2.3850  &  3.0885   & 4.2495 \\
            6.0803  &  4.8112   & 3.9878 &   3.5979  &  3.6213
	\end{bmatrix},\\
	\mathcal{GU} &= \begin{bmatrix}
		-17.9542 & -4.2336 & 10.5235 & 22.7041 & 29.3259
	\end{bmatrix},
	\end{aligned}
    $$
    with which Assumption \ref{assumption full row rank} is satisfied. Considering the physical meanings of the state variables, it is natural to assume that the $(1,1)$-entry of $A_{\textnormal{true}}$ is known to be $1$. Besides, we assume that all zero entries of both $A_{\textnormal{true}}$ and $B_{\textnormal{true}}$ are known, while the remaining nonzero entries are unknown. As such, in this example, the index set $s = \{1,2,3,4,9,10,11\}$, $v_1=1$ and $v_i=0$ for all $i \in s \setminus \{1\}$. Moreover, we have $\bar{s} = \{5,6,7,8,12\}$. 
    
    We now apply Theorem \ref{theorem control design} to generate a controller. We take $a(x)$ to be a positive constant and $c(x)$ to be a polynomial of degree at most $5$, and then impose the restriction that the constant term of $c(x)$ is zero, ensuring $c(0) = 0$. Let $b(x) = x_1^2+x_1x_2+x_2^2+(x_1+x_2+x_3-x_2^3)^2$, $\epsilon = 10^{-4}$ and $\beta(x) = 10^{-4}(x_1^2 + x_2^2 + x_3^2)$. Condition \ref{theorem control design ii} in Theorem \ref{theorem control design} is satisfied by choosing $\alpha = 4$. Define $R(x)$ as in \eqref{Rx}. We formulate an SOS program imposing constraints \ref{theorem control design iv} and \ref{theorem control design v} in Theorem \ref{theorem control design}, and we solve this for $a(x)$ and $c(x)$. We obtain $a(x) = 0.1440$ and
    $$
        \begin{aligned}
            c(x) = &-0.1776 x_1 -0.2387 x_2-0.2079 x_3\\
            &  -0.0033  x_1x_2^2 + 0.2727 x_2^3 + 0.4288 x_2^2x_3\\
            &  -0.4287 x_2^5.
        \end{aligned}
    $$
    Since the controller $K(x) = \frac{c(x)}{a(x)}$, we have
    \begin{equation} \label{example DC Kx}
	\begin{aligned}
            K(x) = &-1.2333 x_1 -1.6576 x_2-1.4438 x_3\\
            &  -0.0229  x_1x_2^2 +1.8938 x_2^3 +2.9778 x_2^2x_3\\
            &  -2.9771 x_2^5.
	\end{aligned}
    \end{equation}
    Subsequently, given this controller $K(x)$, we apply Corollary \ref{corollary global} to verify the global asymptotic stability of the closed-loop system for all systems compatible with the data and the prior knowledge. We take $V$ to be a polynomial of degree at most 6. To ensure $V(0) = 0$, we take the constant term of $V$ to be zero. Define $L(x)$ as in \eqref{Lx}. Let $\epsilon(x) = \beta(x) = 10^{-4}(x_1^2+x_2^2 + x_3^2)$. Finally, we formulate an SOS program that imposes constraints \ref{corollary global c} and \ref{corollary global d} in Corollary \ref{corollary global}, and we solve this for $V$. We obtain a radially unbounded function $V \in \mathcal{V}$ of degree 6. This verifies the asymptotic stability of all closed-loop systems obtained by interconnecting any system compatible with the data and the prior knowledge with the controller \eqref{example DC Kx}.
    \begin{figure}[t]
	\begin{center}
        \includegraphics[trim=50 0 0 10,clip, width=9cm]{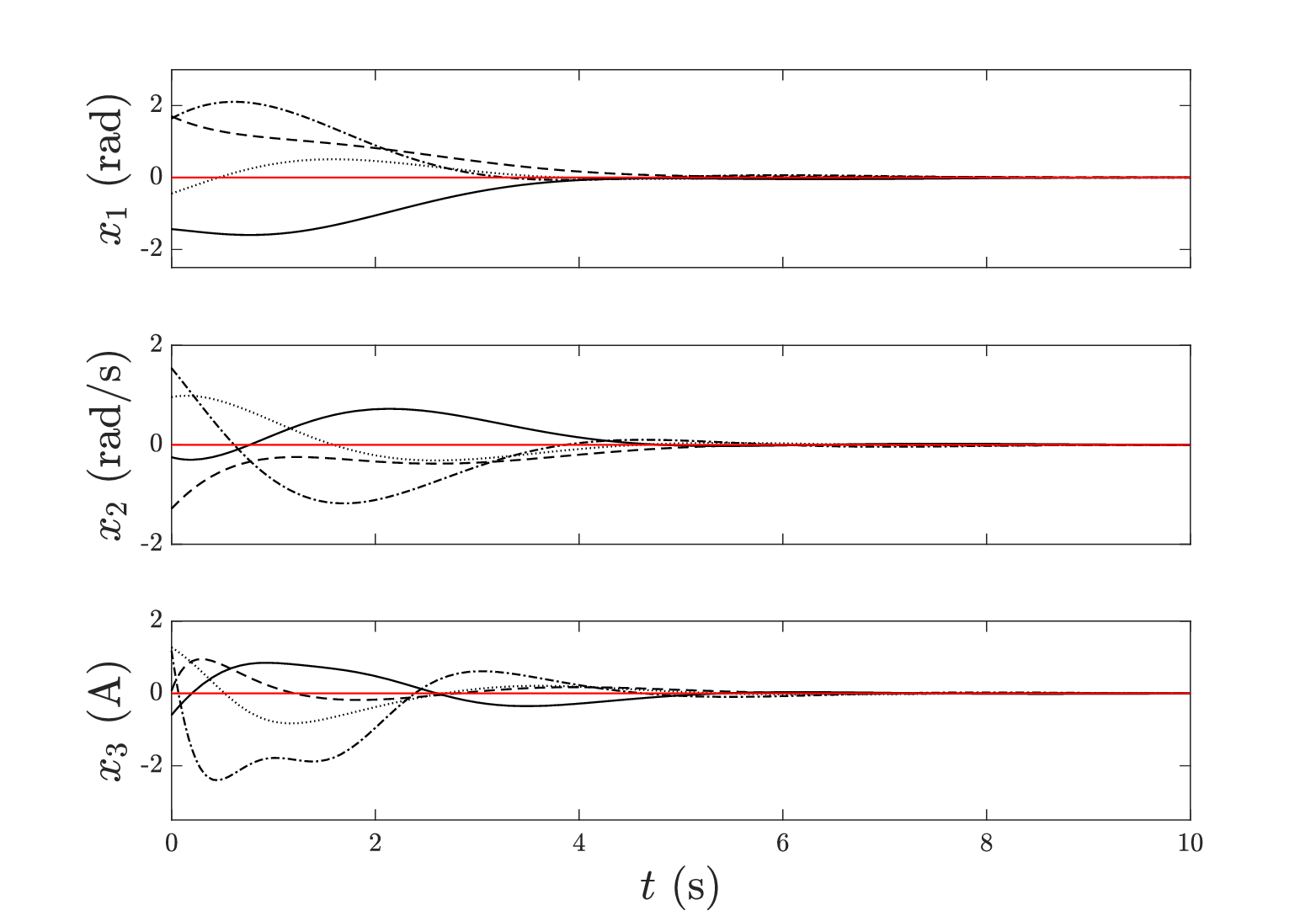}
        \caption{Trajectories of the true closed-loop system for 4 different initial states over 10 seconds. The zero state in each subplot is depicted by a red line.}
        \label{fig DC}
        \end{center}
    \end{figure}

    In Fig.~\ref{fig DC}, we present 4 sets of trajectories of the true closed-loop system over a duration of 10 seconds. Each entry of the initial state $x(0)$ is chosen independently and uniformly at random from the interval $[-2,2]$. These figures further illustrate that the origin of the true closed-loop system is asymptotically stable.
\end{example}

\section{Proofs} \label{section Proofs}
\subsection{A specialized S-lemma}    
The proofs of Theorems \ref{theorem control design} and \ref{theorem DDC strict} mainly rely on a specialized version of the S-lemma. To present this result, we first introduce some terminology. For $\Pi \in \mathbb{S}^{1+r}$, define
$$
    \mathcal{Z}_r(\Pi) := \left\{z \in \mathbb{R}^r \mid \begin{bmatrix}
		1 \\ z
	\end{bmatrix}^{\top} \Pi \begin{bmatrix}
		1 \\ z
	\end{bmatrix} \geq 0 \right\},
$$
and define the set of matrices
$$
\begin{aligned}
    \mathbf{\Pi}_{1,r}&: = \\ &\left\{\begin{bmatrix}
		\Pi_{11} & \Pi_{12}\\\Pi_{21} & \Pi_{22}
	\end{bmatrix} \in \mathbb{S}^{1+r} \mid \Pi_{22} < 0\text{ and } \Pi | \Pi_{22} \geq 0 \right\},
\end{aligned}
$$
where we recall that $\Pi | \Pi_{22} = \Pi_{11}-\Pi_{12}\Pi_{22}^{-1}\Pi_{21}$ is the Schur complement of $\Pi$ with respect to $\Pi_{22}$. Then, let $N \in \mathbf{\Pi}_{1,r}$. It follows from \cite[Thm~3.3]{DDCQMI} that for the matrix $N \in \mathbf{\Pi}_{1,r}$, $z \in \mathcal{Z}_r(N)$ if and only if 
\begin{equation} \label{ast}
    z = -N_{22}^{-1}N_{21}+(-N_{22})^{-\frac{1}{2}}s(N|N_{22})^{\frac{1}{2}}
\end{equation}
for some $s \in \mathbb{R}^r$ such that $\|s\| \leq 1$. In the following lemma, we state a specialized S-lemma \cite{SurveySlemma}, involving one linear and one quadratic inequality, as opposed to two quadratic inequalities. 

\begin{lemma} \label{lemma nonstrict slemma}
Let $N \in \mathbf{\Pi}_{1,r}$, $a \in \mathbb{R}$, and $\lambda \in \mathbb{R}^r$. Then, 
\begin{equation} \label{lemma a}
    \lambda^{\top} z + a \geq 0 \quad \forall z \in \mathcal{Z}_r(N)
\end{equation}
if and only if
\begin{equation} \label{lemma b}
    \begin{bmatrix}
        -\lambda^{\top}N_{22}^{-1}N_{21} + a & \lambda^{\top}(N|N_{22})^{\frac{1}{2}} \\
	(N|N_{22})^{\frac{1}{2}}\lambda & (\lambda^{\top}N_{22}^{-1}N_{21}-a)N_{22}
    \end{bmatrix} \geq 0.
\end{equation}
\end{lemma}
\begin{proof}
    In case $\lambda = 0$, the statement readily follows since $N_{22} < 0$. Suppose that $\lambda \neq 0$. As $N \in \mathbf{\Pi}_{1,r}$, using \eqref{ast}, \eqref{lemma a} is equivalent to
    $$
		-\lambda^{\top}N_{22}^{-1}N_{21}+\lambda^{\top}(-N_{22})^{-\frac{1}{2}}s(N|N_{22})^{\frac{1}{2}} +a \geq 0,
    $$
    for all $s \in \mathbb{R}^{r}$ with $\|s\| \leq 1$. The latter statement is equivalent to	
    \begin{equation} \label{lemma max}
        \begin{aligned}
            -\lambda^{\top}N_{22}^{-1}N_{21} + a \geq  \max_{\|s\| \leq 1} \left(-\lambda^{\top}(-N_{22})^{-\frac{1}{2}}s(N|N_{22})^{\frac{1}{2}}\right).
        \end{aligned}
	\end{equation}
	We claim that
	\begin{equation}\label{lemma claim}
        \begin{aligned}
            &\max_{\|s\| \leq 1} \left(-\lambda^{\top}(-N_{22})^{-\frac{1}{2}}s(N|N_{22})^{\frac{1}{2}}\right)\\ =&  (N|N_{22})^{\frac{1}{2}} \|(-N_{22})^{-\frac{1}{2}}\lambda\|.
        \end{aligned}
	\end{equation}
	To prove \eqref{lemma claim}, we first recall that $\lambda \neq 0$ and $N_{22} < 0$, and hence $(-N_{22})^{-\frac{1}{2}}\lambda \neq 0$. Since $N|N_{22} \geq 0$, by using Cauchy-Schwarz's inequality, for all $s \in \mathbb{R}^r$ satisfying $s^{\top}s \leq 1$,
    $$
		-\lambda^{\top}(-N_{22})^{-\frac{1}{2}}s(N|N_{22})^{\frac{1}{2}} \leq (N|N_{22})^{\frac{1}{2}}\|(-N_{22})^{-\frac{1}{2}}\lambda\|.
    $$
    Define
    $$
		\bar{s}:=\frac{-(-N_{22})^{-\frac{1}{2}}\lambda}{\|(-N_{22})^{-\frac{1}{2}}\lambda\|}.
    $$
    It is clear that $\bar{s} \in \mathbb{R}^r$ and $\bar{s}^{\top}\bar{s} \leq 1$. Moreover, we have
    $$
		-\lambda^{\top}(-N_{22})^{-\frac{1}{2}}\bar{s}(N|N_{22})^{\frac{1}{2}} = (N|N_{22})^{\frac{1}{2}}\|(-N_{22})^{-\frac{1}{2}}\lambda\|.
    $$
    Therefore, we conclude that \eqref{lemma claim} holds. Consequently, \eqref{lemma max} holds if and only if
    $$
		-\lambda^{\top}N_{22}^{-1}N_{21}+a \geq (N|N_{22})^{\frac{1}{2}} \|(-N_{22})^{-\frac{1}{2}}\lambda\|.
    $$
    Equivalently, $-\lambda^{\top}N_{22}^{-1}N_{21} + a \geq 0$ and
    \begin{equation} \label{lemma square}
		(-\lambda^{\top}N_{22}^{-1}N_{21}+a)^2 \geq (N|N_{22}) \lambda^{\top}(-N_{22})^{-1}\lambda.
    \end{equation}
    Using the latter, it is straightforward to see that \eqref{lemma a} and \eqref{lemma b} are equivalent in case $-\lambda^{\top}N_{22}^{-1}N_{21} + a = 0$. Next, consider the case that $-\lambda^{\top}N_{22}^{-1}N_{21} + a > 0$. The inequality \eqref{lemma square} can then be rewritten as
    \begin{equation} \label{lemma Schur}
        \begin{aligned}
            &(-\lambda^{\top}N_{22}^{-1}N_{21}+a )\\ &- (-\lambda^{\top}N_{22}^{-1}N_{21}+a)^{-1}(N|N_{22}) \lambda^{\top}(-N_{22})^{-1}\lambda \geq 0.
         \end{aligned}
    \end{equation}
    Finally, since $N_{22} < 0$, by using a Schur complement argument, \eqref{lemma Schur} is equivalent to \eqref{lemma b}. This proves the lemma.
\end{proof}

Note that Lemma \ref{lemma nonstrict slemma} considers non-strict inequalities. In what follows, we extend this lemma by addressing strict inequalities. The following theorem is one of the main technical results of this paper. 
\begin{theorem} \label{theorem state-independent}
Let $N \in \mathbf{\Pi}_{1,r}$, $a \in \mathbb{R}$ and $\lambda \in \mathbb{R}^r$. The following statements are equivalent:
\begin{enumerate}[label=(\roman*)]
	\item \label{theorem state-independent a} $	\lambda^{\top} z + a> 0\quad \forall z \in \mathcal{Z}_r(N)$.
	\item \label{theorem state-independent b} $\begin{bmatrix}
		-\lambda^{\top}N_{22}^{-1}N_{21} + a & \lambda^{\top}(N|N_{22})^{\frac{1}{2}} \\
		(N|N_{22})^{\frac{1}{2}}\lambda & (\lambda^{\top}N_{22}^{-1}N_{21} - a)N_{22}
	\end{bmatrix} > 0$.
	\item \label{theorem state-independent c} There exists a scalar $\beta>0$, such that 
	\begin{equation} \label{theorem c}
		\begin{bmatrix}
			-\lambda^{\top}N_{22}^{-1}N_{21} + a - \beta & \lambda^{\top}(N|N_{22})^{\frac{1}{2}} \\
			(N|N_{22})^{\frac{1}{2}}\lambda & (\lambda^{\top}N_{22}^{-1}N_{21} - a)N_{22}
		\end{bmatrix} \geq 0.
	\end{equation}
\end{enumerate}
\end{theorem}

\begin{proof}
As $N \in \mathbf{\Pi}_{1,r}$, the equivalence is straightforward if $\lambda = 0$. Therefore, we only consider the case that $\lambda$ is nonzero. Since it is obvious that \ref{theorem state-independent b} implies \ref{theorem state-independent c}, our strategy is to first show that \ref{theorem state-independent a} implies \ref{theorem state-independent b}, and then that \ref{theorem state-independent c} implies \ref{theorem state-independent a}. Since the set $\mathcal{Z}_r(N)$ is closed and bounded, \ref{theorem state-independent a} holds if and only if there exists a scalar $\beta > 0$ such that 
\begin{equation} \label{theorem a beta}
	\lambda^{\top} z + a - \beta\geq 0 \quad \forall z \in \mathcal{Z}_r(N).
\end{equation}
According to Lemma 1, \eqref{theorem a beta} holds if and only if 
\begin{equation} \label{theorem b beta}
	\begin{bmatrix}
		-\lambda^{\top}N_{22}^{-1}N_{21} + a - \beta & \lambda^{\top}(N|N_{22})^{\frac{1}{2}} \\
		(N|N_{22})^{\frac{1}{2}}\lambda & (\lambda^{\top}N_{22}^{-1}N_{21} - a + \beta)N_{22}
	\end{bmatrix} \geq 0.
\end{equation}
As $\beta > 0$ and $N_{22} < 0$, \eqref{theorem b beta} implies that \ref{theorem state-independent b} holds.\\
Finally, we prove that \ref{theorem state-independent c} implies \ref{theorem state-independent a}. Suppose \ref{theorem state-independent c} holds. By using a Schur complement argument, \ref{theorem state-independent c} implies that $-\lambda^{\top}N_{22}^{-1}N_{21} + a > 0$ and
$$
\begin{aligned}
    &(-\lambda^{\top}N_{22}^{-1}N_{21}+a )\\ 
    &- (-\lambda^{\top}N_{22}^{-1}N_{21}+a)^{-1}(N|N_{22}) \lambda^{\top}(-N_{22})^{-1}v > 0.
\end{aligned}
$$
This implies that
$$
	-\lambda^{\top}N_{22}^{-1}N_{21} + a > (N|N_{22})^{\frac{1}{2}} \|(-N_{22})^{-\frac{1}{2}}\lambda\|.
$$
Since \eqref{lemma claim} holds, 
$$
    -\lambda^{\top}N_{22}^{-1}N_{21}+\lambda^{\top}(-N_{22})^{-\frac{1}{2}}s(N|N_{22})^{\frac{1}{2}} +a > 0,
$$
for all $s \in \mathbb{R}^{r}$ such that $\|s\| \leq 1$. By \eqref{ast}, this implies \ref{theorem state-independent a}, which proves the theorem.
\end{proof}

\subsection{Proof of Theorem \ref{theorem control design}} \label{section Proof of Theorem 2}
Define $\rho (x)$ as in \eqref{rhox}. Since $b(x)>0$ for all $x \in \mathbb{R}^n \setminus \{0\}$, condition \ref{theorem control design iv} guarantees $\rho(x) > 0$ for all $x \in \mathbb{R}^n \setminus \{ 0 \}$. Condition \ref{theorem control design ii} ensures that $\rho(x)(AF(x) + BG(x)K(x))/\|x\|$ is integrable on $\{ x \in \mathbb{R}^n \mid \|x\| \geq 1 \}$ for all $(A,B) \in \Sigma$. Since $c(0) = 0$ and $a(x)>0$ for all $x \in \mathbb{R}^n$, the controller $K(x) = \frac{c(x)}{a(x)}$ satisfies $K(0) = 0$. Combining this with $F(0) = 0$, the origin is an equilibrium point of the closed-loop system \eqref{closedloopsystem} for all $(A,B) \in \Sigma$.

Subsequently, since $N \in \mathbb{S}^{1+\ell}$, $\mathcal{Z}_{\ell}(N)$ denotes the set of all vectors satisfying \eqref{bound vsbar}. Recall that under Assumption \ref{assumption full row rank}, the partitioned matrix $N$ satisfies $N_{22} < 0$ and $N|N_{22} \geq 0$. Therefore, we have $N \in \mathbf{\Pi}_{1,\ell}$. In what follows, let $x$ be fixed in $\mathbb{R}^n\setminus\{0\}$. Clearly, we have $\beta(x)>0$. By taking $\lambda = R_{\bar{s}}(x)$ and $a = R_{s}^{\top}(x) v_{s}$, it follows from Theorem \ref{theorem state-independent} that \eqref{theorem control design v SOS} implies \eqref{Rxv tau} holds for all $v_{\bar{s}} \in \mathcal{Z}_{\ell}(N)$. Since this implication holds for all $x \in \mathbb{R}^n \setminus \{ 0 \}$, we have that the inequality \eqref{Rxv tau} holds for all $x \in \mathbb{R}^n \setminus \{ 0 \}$ and all $v_{\bar{s}} \in \mathcal{Z}_{\ell}(N)$.
	
Finally, since all conditions in Theorem \ref{theorem density} are satisfied for all $(A,B) \in \Sigma$, we conclude that Theorem \ref{theorem control design} follows from Theorem \ref{theorem density}. \hfill$\square$

\subsection{Proof of Theorem \ref{theorem DDC strict}} \label{section Proof of Theorem 3}
Given a continuous $K: \mathbb{R}^n \rightarrow \mathbb{R}^m$ satisfying $K(0) = 0$, it follows from the equivalence between \eqref{oridotVx} and \eqref{Lxv tau} that the data $(\dot{\mathcal{X}}, \mathcal{X}, \mathcal{U})$ are informative for closed-loop stability if and only if there exists a radially unbounded $V \in \mathcal{V}$ such that \eqref{Lxv tau} holds for all $v_{\bar{s}} \in \mathcal{Z}_{\ell}(N)$. 

In what follows, let the function $V \in \mathcal{V}$ be fixed. Recall that $N \in \mathbf{\Pi}_{1, \ell}$, as explained in the proof of Theorem \ref{theorem control design}. We apply Theorem \ref{theorem state-independent} for all $x\in \mathbb{R}^n \setminus\{ 0\}$ by taking $\lambda = L_{\bar{s}}(x)$ and $a = L_{s}^{\top}(x) v_{s}$. Then, we have that \eqref{Lxv tau} holds for all $v_{\bar{s}} \in \mathcal{Z}_{\ell}(N)$ if and only if there exists a function $\beta: \mathbb{R}^n \rightarrow \mathbb{R}$ such that for all $x \in \mathbb{R}^{n} \setminus \{0\}$, $\beta(x) > 0$ and \eqref{theorem DDC strict b matrix inequality} holds. Since this result applies to any $V \in \mathcal{V}$, we conclude that there exists a radially unbounded $V \in \mathcal{V}$ such that \eqref{Lxv tau} holds for all $v_{\bar{s}} \in \mathcal{Z}_{\ell}(N)$ if and only if there exist a radially unbounded $V \in \mathcal{V}$ and a $\beta: \mathbb{R}^n \rightarrow \mathbb{R}$ such that for all $x\in \mathbb{R}^n \setminus\{ 0\}$, $\beta(x) > 0$, and \eqref{theorem DDC strict b matrix inequality} holds. \hfill$\square$

\subsection{Proof of Proposition \ref{proposition comparison}}
\label{section Proof of Proposition 4}
To prove the ``if" part, assume that there exist a radially unbounded $V \in \mathcal{V}$ and $\gamma,\ \beta: \mathbb{R}^n \rightarrow \mathbb{R}$ such that for all $x \in \mathbb{R}^n \setminus \{ 0\}$, $\gamma(x) \geq 0$, $\beta (x) > 0$ and \eqref{theorem DDC strict b1 matrix inequality} holds. Let $(A,B) \in \Sigma$ and let $v_{\bar{s}}$ denote the corresponding vector satisfying \eqref{bound vsbar}. Multiply the inequality \eqref{theorem DDC strict b1 matrix inequality} from left by
$
    \begin{bmatrix}
        1 & v_{\bar{s}}^{\top}
    \end{bmatrix}
$
and from right by its transpose. This yields \eqref{Lxv tau}. In other words, the data are informative for closed-loop stability.

To prove the ``only if" statement, assume that the data $(\dot{\mathcal{X}}, \mathcal{X},\mathcal{U})$ are informative for closed-loop stability. By Theorem \ref{theorem DDC strict}, there exist a radially unbounded $V \in \mathcal{V}$ and a $\beta: \mathbb{R}^n \rightarrow \mathbb{R}$ such that for all $x\in \mathbb{R}^n \setminus\{ 0\}$, $\beta(x) > 0$, and \eqref{theorem DDC strict b matrix inequality} holds, It follows from the $(1,1)$-entry of the matrix in \eqref{theorem DDC strict b matrix inequality} that
$$
    L_{\bar{s}}^{\top}(x)N_{22}^{-1}N_{21} - L_{s}^{\top}(x) v_{s} < 0\quad \forall x\in \mathbb{R}^n \setminus\{ 0\}. 
$$
Define
$$
    \gamma(x) := \frac{L_{\bar{s}}^{\top}(x)N_{22}^{-1}L_{\bar{s}}(x)}{L_{\bar{s}}^{\top}(x)N_{22}^{-1}N_{21} - L_{s}^{\top}(x) v_{s}}.
$$
Since $N_{22} < 0$, this $\gamma(x)$ clearly satisfies
$$
    \gamma (x) \geq 0\quad \forall x\in \mathbb{R}^n \setminus\{ 0\}.
$$
For all $x\in \mathbb{R}^n \setminus\{ 0\}$ satisfying $L_{\bar{s}}(x) = 0$, we have $\gamma (x) = 0$ and
$$
	L_{s}^{\top}(x) v_{s} > 0. 
$$
Therefore, there exists a function $\beta: \mathbb{R}^n \rightarrow \mathbb{R}$ such that for all $x\in \mathbb{R}^n \setminus\{ 0\}$, $\beta(x)>0$ and the matrix inequality \eqref{theorem DDC strict b1 matrix inequality} holds. Subsequently, we consider the case that $x\in \mathbb{R}^n \setminus\{ 0\}$ is such that $L_{\bar{s}}(x) \neq 0$. We first use a Schur complement argument on the matrix inequality \eqref{theorem DDC strict b matrix inequality}, which yields
\begin{equation}\label{proposition 1 substitute}
\begin{aligned}
    -L_{\bar{s}}^{\top}(x)&N_{22}^{-1}N_{21} + L_{s}^{\top}(x) v_{s}\\
    &+\frac{(N|N_{22})\left( L_{\bar{s}}^{\top}(x) N_{22}^{-1}L_{\bar{s}}(x)\right)}{-L_{\bar{s}}^{\top}(x)N_{22}^{-1}N_{21} + L_{s}^{\top}(x) v_{s}} > 0.
\end{aligned}
\end{equation}
Since $\gamma (x) > 0$ for all nonzero $x$ satisfying $L_{\bar{s}}(x) \neq 0$, \eqref{proposition 1 substitute} is equivalent to
    $$
    \begin{aligned}
        -2L_{\bar{s}}^{\top}(x)N_{22}^{-1}N_{21} + 2L_{s}^{\top}(x) v_{s}- \gamma(x) (N|N_{22})&\\
        + \gamma^{-1}(x) L_{\bar{s}}^{\top}(x)N_{22}^{-1} L_{\bar{s}}(x)& > 0.
    \end{aligned}   
    $$
Then, using a Schur complement argument, we have
$$
    \begin{bmatrix}
        2L_{s}^{\top}(x) v_{s} & L_{\bar{s}}^{\top}(x) \\ L_{\bar{s}}(x) & 0
    \end{bmatrix} - \gamma (x) N > 0.
$$
Therefore, there exists a function $\beta: \mathbb{R}^n \rightarrow \mathbb{R}$ such that for all $x\in \mathbb{R}^n \setminus\{ 0\}$, $\beta(x)>0$ and \eqref{theorem DDC strict b1 matrix inequality} holds. This proves the proposition.  \hfill$\square$

\subsection{Proof of Corollary \ref{corollary global}} \label{section Proof of Corollary global}
Condition \ref{corollary global c} implies that $V \in \mathbb{R}[x]$ and $V(x) \geq \epsilon(x)$ for all $x \in \mathbb{R}^n$. It follows from \ref{corollary global a} and \ref{corollary global b} that $V(0) = 0$, $\epsilon(x) > 0$ for all $x\in \mathbb{R}^n \setminus\{ 0\}$, and therefore $V \in \mathcal{V}$. Moreover, as $\epsilon(x)$ is radially unbounded, $V$ is also radially unbounded. Then, since $a(x) > 0$ for all $x\in \mathbb{R}^n$, condition \ref{corollary global d} implies that \eqref{theorem DDC strict b matrix inequality} holds for all $x\in \mathbb{R}^n \setminus\{ 0\}$. Therefore, we conclude from Theorem \ref{theorem DDC strict} that the data $(\dot{\mathcal{X}},\mathcal{X},\mathcal{U})$ are informative for closed-loop stability with respect to $K$. \hfill$\square$

\subsection{Proof of Corollary \ref{corollary alternative}} \label{section Proof of Corollary alternative}
Suppose that the conditions of Corollary \ref{corollary alternative} hold. Recall that under conditions \ref{corollary global a}, \ref{corollary global b} and \ref{corollary global c} in Corollary \ref{corollary global}, we have that $V \in \mathcal{V}$ and $V$ is radially unbounded. Since $a(x) > 0$ for all $x\in \mathbb{R}^n$, the function $\frac{\beta(x)}{a(x)}$ is positive for all $x \in \mathbb{R}^n \setminus\{ 0\}$. Therefore, by the constraint \eqref{corollary alternative d SOS}, \eqref{eq form 1} holds with $\beta(x)$ replaced by $\frac{\beta(x)}{a(x)}$. Since \eqref{theorem DDC strict b matrix inequality} and \eqref{eq form 1} are equivalent, it follows from Theorem \ref{theorem DDC strict} that the data $(\dot{\mathcal{X}},\mathcal{X},\mathcal{U})$ are informative for closed-loop stability. \hfill$\square$

\section{Conclusion} \label{section Conclusion}
In this study, we have proposed a novel data-driven stabilization method for a class of the polynomial systems using data corrupted by bounded noise. Beyond handling noise, our method can also readily incorporate prior knowledge on the system parameters. The control design consists of two steps. Firstly, by the dual Lyapunov theorem, we have formulated an SOS program to generate a state-feedback controller, ensuring that the trajectories of all closed-loop systems compatible with the data and the prior knowledge converge to zero for \textit{almost all} initial states. Secondly, we have verified the asymptotic stability of these closed-loop systems using a common Lyapunov function. In both steps, our methods mainly rely on a specialized version of the S-lemma that provides conditions under which a linear inequality is implied by a quadratic one. Compared to the classical S-lemma \cite{SurveySlemma}, our version avoids the introduction of multipliers. This implies that SOS programming for data-driven stabilization does not require the parameterization of multipliers, which leads to clear benefits from a computational point of view.

The main advantage of our method is twofold. Firstly, our method delivers stabilizing controllers in situations where existing approaches, like \cite{DDCmiddleCDC,DDCmiddleTAC}, fail to do so. Secondly, our method requires lower computational effort than the one in \cite{DDCDensity}. However, a potential drawback of the proposed methodology stems from the use of density functions, whose denominator has to be chosen a priori. This can be challenging for high-dimensional systems.

For future research, one possible direction is to extend our results to data-driven optimal control problems, such as minimizing the required input energy. Another open question is the data-driven regulation problem for polynomial systems, which aims at designing controllers in such a way that the system state tracks a given reference signal. Finally, extending our approach to the discrete-time case, based on a counterpart of density functions for discrete-time systems, would also be an interesting direction for future work.

\bibliographystyle{plain}       
\bibliography{ReferencePapers}      

\end{document}